\documentclass[11pt, a4paper]{article}

\RequirePackage[OT1]{fontenc}
\RequirePackage{amsthm,amsmath}
\RequirePackage[numbers]{natbib}
\usepackage{changepage}
\usepackage[english]{babel}
\usepackage[pdftex]{graphicx}
\usepackage{float}
\usepackage{amsfonts}
\usepackage{enumerate}
\usepackage{amsmath}
\usepackage{comment}
\usepackage{authblk}
\usepackage{mathtools}
\mathtoolsset{showonlyrefs}			%numera solo le equazioni richiamate con eqref
\numberwithin{equation}{section}
\theoremstyle{plain}
\newtheorem{thm}{Theorem}[section]

\newtheorem{prop}[thm]{Proposition}

\newtheorem{lem}[thm]{Lemma}
\theoremstyle{definition}
\newtheorem{dfn}[thm]{Definition}

\theoremstyle{remark}
\newtheorem{rem}[thm]{Remark}

\newcommand{\probSpace}{(\Omega,\mathcal{F},\{\mathcal{F}_{t}\}_{t\geq 0},\mathbb{P})}
\def\R{\mathbb R}
\def\N{\mathbb N}
\def\Z{\mathbb Z}

\def\P{\mathbb P}

\def\1{\mathds{1}}
\def\ra{\rightarrow}

\def\ve{\varepsilon}

\def\s{\sigma}
\def\t{\theta}

\def\T{\Theta}
\def\piq{\frac{\pi}{4}}

\def\pim{\frac{\pi}{2}}
\def\pitm{\frac{3\pi}{2}}

\usepackage{geometry}
\usepackage{todonotes}

\title{Strong Solutions to SDEs with Supercritical Drift arising in Navigation Models}
\author[1]{Carlo Ciccarella}

\affil[1]{Institut de Mathématiques, Ecole Polytechnique Fédérale de Lausanne, Station 8, CH-1015, Lausanne, Switzerland}

\begin{document}

  \maketitle

\begin{abstract}

We prove strong existence and pathwise uniqueness for two stochastic models of a
seeker steering toward a target, written in polar coordinates. In both, the angular drift carries a $\frac1{r}$-type singularity which belongs to the supercritical regime in $\R^2$.
Standard results for SDEs with singular
drift therefore do not apply, and we give a new proof of strong well-posedness  based on a pathwise argument. 

The two models arise from sailboat navigation and proportional navigation.
We study the limiting regime in which the stopping radius around the target tends to
zero and prove that, despite the singularity at the origin, each system admits a
unique strong solution up to the hitting time of the target. 

These results provide an example of strong well-posedness in a regime where the
general theory does not apply.
\end{abstract}

%We extend a previously studied optimal control problem for a sailboat aiming to reach an upwind buoy, originally formulated with the target defined as a ball of radius $\eta > 0$, to the case where the target is a single point located at the origin. The control problem is considered under the assumption of zero  tacking cost. 
%The additional difficulty is that the process controlled by the guessed optimal strategy $\sa$, which it is expressed in polar coordinates, has a drift that is unbounded and not integrable in the Lebesgue sense near the origin. Moreover, $\sa$ may involve an infinite number of tacks in any compact set, making  the drift  not even continuous. In particular,  it belongs to the  supercritical space $L^q([0,T], L^p_{loc}(\R^2))$ with $q <1$, resulting in $\frac{d}{p} + \frac{2}{q} >1$, where $d=2$, since the boat moves in $\R^2$.   Standard results on the existence and uniqueness of  strong solutions do not apply,  so we provide a novel proof in this degenerate case. 



%\begin{here}%[class=MSC]
%\kwd[Primary ]{49L12, 60H10, 93B52, 93E20}
%\kwd{37A50}
%\kwd[; secondary ]{37A50}
%{\bf Key words and phrases:} stochastic control; verification theorem;

%\end{here}

%\tableofcontents

\section{Introduction}

We prove strong well-posedness to stochastic differential equations with supercritical drift naturally arising in navigation   models. 

The general framework is a seeker steering toward a target in a noisy environment, where the noise affects the angular coordinate. We treat two instances:  the \textit{sailboat navigation}  \cite{CDV2025, CC, DalangDumasEtAl2015, Vinckenbosch2012} and \textit{proportional navigation} models, the latter widely used in missile guidance  \cite{Berglund2001, shneydor1998, zarchan2012}. 

In $\R^2$, with the target at the origin, the dynamics in polar coordinates are described by a system of stochastic differential
equations whose angular drift carries a
$\frac 1r$-type singularity.

Let $\probSpace$ be a filtered probability space supporting a standard Brownian
motion $(B_t)_{t\geq0}$. The models we consider are described by a system of the form
\begin{equation}\label{state}
\begin{cases}
dR_t = \mu_1(R_t, \Theta_t)\, dt, & R_0 = r,\\
d\T_t = \mu_2(R_t, \T_t)\, dt + \sigma\, dB_t, & \T_0 = \t,
\end{cases}
\end{equation}
where the angular drift has the form $\mu_2(r,\theta) = \frac{f(\theta)}{r}$  for a
bounded, piecewise-smooth $f$ depending on the model (see Section~\ref{motiv}).
We refer the reader to Section~\ref{motiv} for a formulation and discussion of the models we consider together with the explanation of why the noise only affects the angular component $\T_t$.

Our main results, Theorem~\ref{Solution_sail} and Theorem~\ref{Solution_pn},
establish strong existence and pathwise uniqueness for the two models up to
the hitting time $\tau$ of the target, where the solution is understood up
to and including $\tau$: the equations hold in the usual strong sense on $[0,t]$
for every $t<\tau$, and the trajectory extends continuously to its almost sure
limit at $\tau$ (the precise notion is Definition~\ref{def:sol_tau}).

Our contribution is  twofold. First, we develop a pathwise method that yields strong existence and pathwise uniqueness in a supercritical regime of type $1/r$ where standard results do not apply. The method in fact covers a broader class of drift and in particular those with $\mu_2(r,\t) =  \frac{f(\t)}{r^\gamma}$, for any  $\gamma \geq 1$ (see Remark \ref{rem_gamma}). 
 Second, we give the first
rigorous treatment of both navigation models as stochastic differential equations, well-posed up to the
hitting time of the target.

Two features place \eqref{state} outside the classical theory. Near the origin
the drift becomes singular, and the diffusion is degenerate, since noise acts
only on the angular component. Classical existence and uniqueness hold under
uniform Lipschitz continuity of the coefficients, and---more
generally---\cite{veretennikov1981strong} established strong well-posedness for
measurable drift of linear growth with nondegenerate Lipschitz diffusion.
Neither covers \eqref{state}.

The singularity alone already lies beyond the standard theory of singular-drift
SDEs. Recall the \textit{subcritical} criterion of Krylov--R\"ockner
\cite{KrylovRockner2005}: writing $L^q([0,T];L^p(\mathbb R^d))$ for the usual
space with norm
\[
\|f\|_{L^q_tL^p_x}
:=
\left(
\int_0^T
\left(
\int_\Omega |f(t,x)|^p\,dx
\right)^{q/p}
dt
\right)^{1/q},
\]
a drift $b\in L^q([0,T];L^p(\mathbb R^d))$ is subcritical in dimension $d=2$
when $\frac{2}{p}+\frac{2}{q}<1$. But in polar coordinates
$dx = r\,dr\,d\theta$, so
\[
\int_0^1\int_0^{2\pi}
\left|\frac{1}{r}\right|^p
r\,dr\,d\theta
=
2\pi\int_0^1 r^{1-p}\,dr
<\infty
\quad\Longleftrightarrow\quad
p< 2,
\]
forcing $2/p>1$ and hence $\frac{2}{p}+\frac{2}{q}>1$ for every
$q\in[1,\infty]$. A $1/r$-type drift in dimension two is therefore
supercritical. Well-posedness in the subcritical regime with nondegenerate
diffusion is due to \cite{KrylovRockner2005}, with a local version in
\cite{Zhang2011}; the critical case $\frac{d}{p}+\frac{2}{q}=1$ was settled in
\cite{nam2018}.

The supercritical regime is substantially harder, and strong well-posedness can
 fail (see \cite[Section~7.4]{Flandoli2019} for a counterexample). The
recent progress here concerns \emph{weak} solutions, and---crucially---relies
on structural or nondegeneracy hypotheses that \eqref{state} does not satisfy.
\cite{HaoZhang2024} obtain weak well-posedness for divergence-free
distributional drifts under nondegenerate noise; our drift is neither
divergence-free nor paired with nondegenerate noise. \cite{SongXie2023} treat
supercritical drifts for $\alpha$-stable SDEs with $\alpha\in\,]0,1]$ and
nondegenerate noise, and \cite{ChaudruDeRaynalMenozzi2022} for
$\alpha\in\,]1,2]$ with Besov drift; the $1/r$ singularity on $\mathbb R^2$ is
not even compatible with their \textit{good condition} at the Brownian endpoint
$\alpha=2$. %In short, the existing supercritical results address weak solutions

On the modeling side, the sailboat system originates in the stochastic control
problem of \cite{CDV2025}, where strong well-posedness was proved until the hitting time of a ball of radius $\eta >0$ around the target. The present work closes the gap by establishing  strong existence
and pathwise uniqueness until the hitting time of the origin i.e. $\eta = 0$. 

The deterministic kinematics of
proportional navigation, including the reduction to a planar system in range and
line-of-sight angle are well known
\cite{shneydor1998,Berglund2001}. To our knowledge, the
Brownian-driven counterpart of these kinematics has not been analyzed
as a stochastic differential equation, and its strong well-posedness is new.

%Our contribution is thus twofold. First, we develop a pathwise method that yields strong existence and pathwise uniqueness in a supercritical regime of type $1/r$ where standard results do not apply. Second, we give the first rigorous treatment of both navigation models as stochastic differential equations, well-posed up to the hitting time of the target.

%\begin{rem}
%Observe that our results hold with the more general in the models we present below, $\gamma = 1$. Nonetheless,  we treat the general case $\gamma \geq 1$ both  because it signifies a consistent departure from the subcritical regime and our methodology could be applied to handle more general types of potential that may appear in physical applications like the gravitational one ($ 1/r^2$).
%
%\end{rem}

%
%
%In  \cite{ling2019} it is proven that,  in the case of unbounded but subcritical drift, that is $b \in L^q([0,T], L^p(\shq))$ with $\frac{d}{p} + \frac{2}{q} <1$, and non-degenerate diffusion, a strong solution exists and is unique.  A local version, that is for $b \in L^q([0,T], L^p_{loc}(\R^d))$ was obtained in \cite{Zhang2011}. Observe that in \cite{ling2019} $\shq$ is allowed to be a subset of $\R^d$.
%In the current situation, since we are in dimension $2$, $d=2$.
%
%
%In the case of critical drift, that is, $\frac{d}{p} + \frac{2}{q} =1$, \cite{nam2018} proves the existence and uniqueness of a strong solution.
%In general, with supercritical drift, the problem is ill-posed, and  \cite[Sec. 7.4]{Flandoli2019} provides a example where a solution does not exist. 

The paper is organized as follows. In Section 2, we introduce the two motivating models and discuss their similarities and differences. We also recall an existence and uniqueness result for strong solutions of SDEs with degenerate noise and discontinuous coefficients. This result yields well-posedness for \eqref{state} up to the hitting time of a ball of radius \(\eta>0\) centered at the origin. In Section 3, we prove the first main result of the paper: existence and uniqueness of a strong solution to \eqref{state} up to the origin, corresponding to the limiting case \(\eta=0\), with coefficients arising from the sailboat trajectory model \eqref{eq_sail}. Finally, in Section 4, we establish the analogous result for the coefficients associated with the proportional navigation model \eqref{eq_pn}.

%
%The paper is organized as follows. In Section 2, we introduce the two motivating models and discuss the similarities and differences between the corresponding stochastic systems. In Section 3, we recall an existence and uniqueness result for strong solutions of SDEs with degenerate noise and discontinuous coefficients. This result yields well-posedness for \eqref{state} up to the hitting time of a ball of radius \(\eta>0\) centered at the origin. In Section 4, we prove the first main result of the paper: existence and uniqueness of a strong solution to \eqref{state} up to the origin, corresponding to the limiting case \(\eta=0\), for the coefficients arising from the Sailboat Trajectory problem \eqref{eq_sail}. Finally, in Section 5, we establish the analogous result for the coefficients associated with the Proportional Navigation problem \eqref{eq_pn}.

\section{Models}\label{motiv}
 
We present below the two models under consideration in which the system \eqref{state} naturally arises.

The first model stems directly from the missile guidance problem, where one of the most common methods to steer a missile toward a target is so-called \emph{proportional navigation}. 
 Since the pursuer travels at high speed, abrupt changes in its direction are not feasible; however, unlike in the second model, it can point directly at the target.

The second model describes a sailboat attempting to reach a target buoy. The main constraint here is that the sailboat cannot sail directly into the wind—a region known as the \emph{no-go zone}. Consequently, if the target lies upwind, a zig-zag trajectory becomes necessary. Only when the wind rotates sufficiently, so that its direction forms a wide enough angle with the line connecting the boat to the buoy,  the boat can head directly toward its target.

In both models, noise enters the dynamics and may originate from wind fluctuations or, more generally, from measurement errors in the pursuer's direction. However, it affects the equations of motion in distinct ways.  In the proportional navigation problem, noise in $\T$
perturbs the angle between the pursuer's direction and the line-of-sight to the target, thereby requiring continuous trajectory adjustments. In the sailboat navigation problem, by contrast, noise rotates the no-go zone, causing the boat's best heading to deviate either closer to or farther from the buoy.

The fact that the noise acts only on the angular component, for the \textit{sailboat navigation} problem, is justified by the fact that, during the length of a sailing regatta, oscillations in the wind direction are more pronounced than oscillations in the wind speed, at least on certain time scales  (see \cite{DalangDumasEtAl2015}). Concerning the \textit{missile guidance} problem, likewise, given the high speed of a missile, disturbances generated by wind fluctuations are affecting mostly its direction.

\subsection{Proportional Navigation}\label{ss1}
 
Equations of the form~\eqref{state} arise in the kinematic modeling
of pursuit-evasion and missile guidance problems see~\cite{shneydor1998, zarchan2012, lin1991}.

The planar
engagement between a pursuer and a motionless target is described
by two scalar quantities: the distance $r$ from the pursuer to the
target, and the \textit{lead angle} $\theta$, defined as the angle
between the pursuer's velocity vector and the \textit{line of sight}. The \textit{line of sight} (LOS) is
the segment connecting the pursuer to the target; the lead angle
$\theta = 0$ corresponds to the pursuer pointing directly at the
target (head-on), and $\theta = \pm\pi$ to the pursuer pointing
directly away from it.

As discussed in \cite{Berglund2001}, proportional navigation is inspired by classical sailboat navigation: the conceptual idea is that the missile's commanded rotation rate ($d \T$) should be proportional to the  line of sight rotation rate.
The proportionality constant is named \textit{navigation constant} and denoted by $N$ below. 

The relative kinematics under the proportional navigation law are derived for example  in \cite[Eq.~(5.6)]{shneydor1998} and reduce to an autonomous system in $(r, \t)$, where $r$ is the range and $\t$ is the lead angle (in \cite{shneydor1998} notations, $\t = \delta$):
\begin{equation}\label{eq_pn0}
 \mu_1(\t)= 
-v \cos(\theta),
\quad \mu_2(r,\t)= 
 -(N-1)\frac{v \sin(\theta)}{r},
\end{equation}
where $N\geq 2$ is the navigation constant.%, that is the  proportionality constant  between the missile's commanded rotation rate (or lateral acceleration) and the measured Line-Of-Sight rotation rate.

If $N=2$, this boils down to a constant acceleration, resulting in an arc-of-circle trajectory in absence of noise:  $\frac{r_0}{r} = \frac{\sin(\t_0 + \t)}{\sin \t_0}$.

If $|\Theta_t| > \pim$, the pursuer will move away from the target in the first part of the trajectory, see left hand side of Fig. \ref{fig_drift_pn}. 
For modeling tractability, we assume instead that the pursuer can reverse its thrust on the rear hemisphere $\t \in [\pim, \pitm]$; this amounts to identifying $\t$ with $\t - \pi$ on that range, giving $\pi$-periodic dynamics. In the left hand side of Fig. \ref{fig_drift_pn}, this corresponds to the pursuer starting from  $\t_0 = 120 ^\circ$ (resp. $\t_0 =  150^\circ$) to reverse its thrust to effectively start at $\t_0 = -60 ^\circ$ (resp. $\t_0 = -30^\circ$).
Thus, the coefficients of \eqref{state} become the following:
\begin{equation}\label{eq_pn}
 \mu_1(\t)= \left\{\begin{array}{l}
-v \cos(\theta),\\ [1ex]
v \cos(\theta), 
\end{array}
\right.
\quad \mu_2(r,\t)= \left\{\begin{array}{cl}
 -(N-1)\dfrac{v \sin(\theta)}{r}, &\qquad \theta \in [-\pim, \pim ],\\[1ex]
 (N-1)\dfrac{v \sin(\theta)}{r}, &\qquad \theta \in ]\pim, \pitm [.
\end{array}
\right.
\end{equation}

In this model, we also consider stochastic effects, modeled here by additive noise on the angular dynamics,
account for sensor measurement errors on the line-of-sight angle and disturbances like strong wind.

\begin{figure}[h]
\centering
\includegraphics[height=5cm, width=5cm,keepaspectratio]{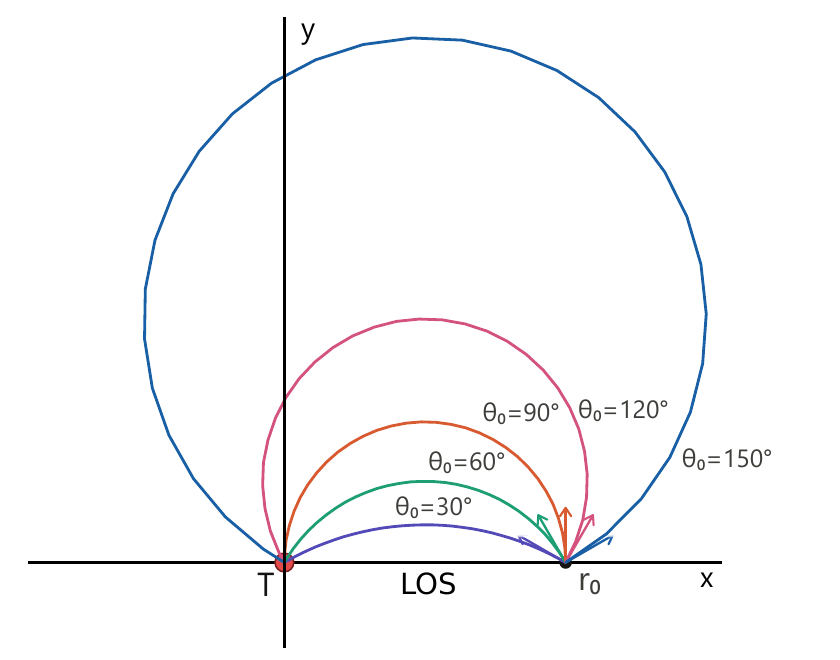}\hspace{1cm}
\includegraphics[height=6cm, width=6cm,keepaspectratio]{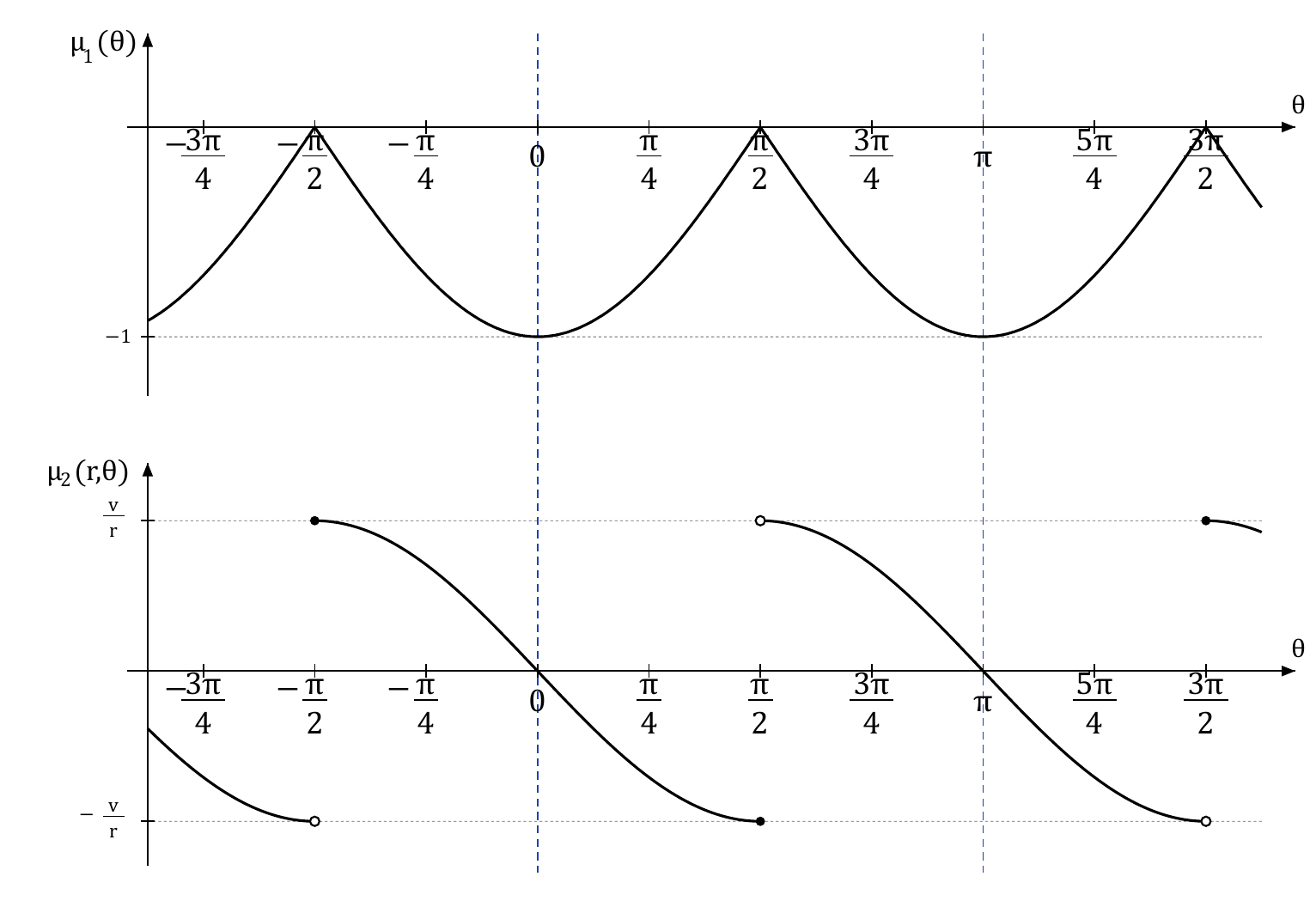}
\caption{Case $N=2$. Left: Trajectories of the pursuer as function of the lead angle $\t$in absence of noise. The line of sight is the segment connecting the target with the pursuer. 
Right: drift $\mu_1$ and $\mu_2$ extended by $2\pi-$periodicity. 
}
\label{fig_drift_pn}
\end{figure}

Extension to a moving target is given in \cite[Eq.~(5.12)]{shneydor1998} and the equations have essentially the same form.

%The first branch is the standard proportional navigation drift; the
%second is active only when the pursuer points into the rear
%hemisphere and rotates $\theta$ toward the front hemisphere, after
%which the standard branch resumes.

\subsection{Sailboat Navigation}\label{ss2}

A second class of problems leading to systems of the form~\eqref{state}
arises in the optimization of sailboat trajectories, as developed in
\cite{CC, CDV2025}; see also
\cite{Vinckenbosch2012,DalangDumasEtAl2015}. In this model, a yacht moves
at constant speed \(v>0\) towards a target buoy, while the wind direction
is noisy and is modeled by a Brownian motion with diffusion coefficient
\(\sigma>0\).

The main constraint is that a sailboat cannot sail directly into the wind, namely inside the so-called \emph{no-go zone} but there has to be a positive angle $\alpha$ (assumed to be $45^\circ$) between its direction and the wind direction. 

We briefly recall the model from \cite{CDV2025}. 
It  assumes a symmetric behavior for upwind and downwind sailing.
For simplicity, the optimal angle is fixed at \(45^\circ\) for an upwind
route and at \(135^\circ\) for a downwind route, the latter meaning that the boat cannot sail with the wind coming directly from the back.  Between these angles,  the yacht can sail  at constant speed $v>0$ in any of the directions within in the butterfly shaded areas in Fig.~\ref{Fig:beating:upwind}.
Each of the two wedges of the butterfly corresponds to the boat sailing on port tack or starboard tack,  depending on whether the wind comes from the left-hand or right-hand side of the boat, respectively. Therefore, for each position and each tack, the direction pointing the closest to the target is chosen, represented by arrows in Fig.~\ref{Fig:beating:upwind}, resulting in two possible choices of direction.

%For each position of the boat and each tack,  the direction which points the closest to the target  is chosen. Therefore for each position of the boat, only two headings are allowed, represented by the arrows in Fig.~\ref{Fig:beating:upwind}, one corresponding to port tack  and the other corresponding to starboard tack.
%The two possible headings are called starboard tack and port tack, depending on whether the wind comes from the right-hand or left-hand side of the boat, respectively.

\begin{figure}[h]
\begin{center}
\includegraphics[height=7cm, width=7cm,
keepaspectratio]{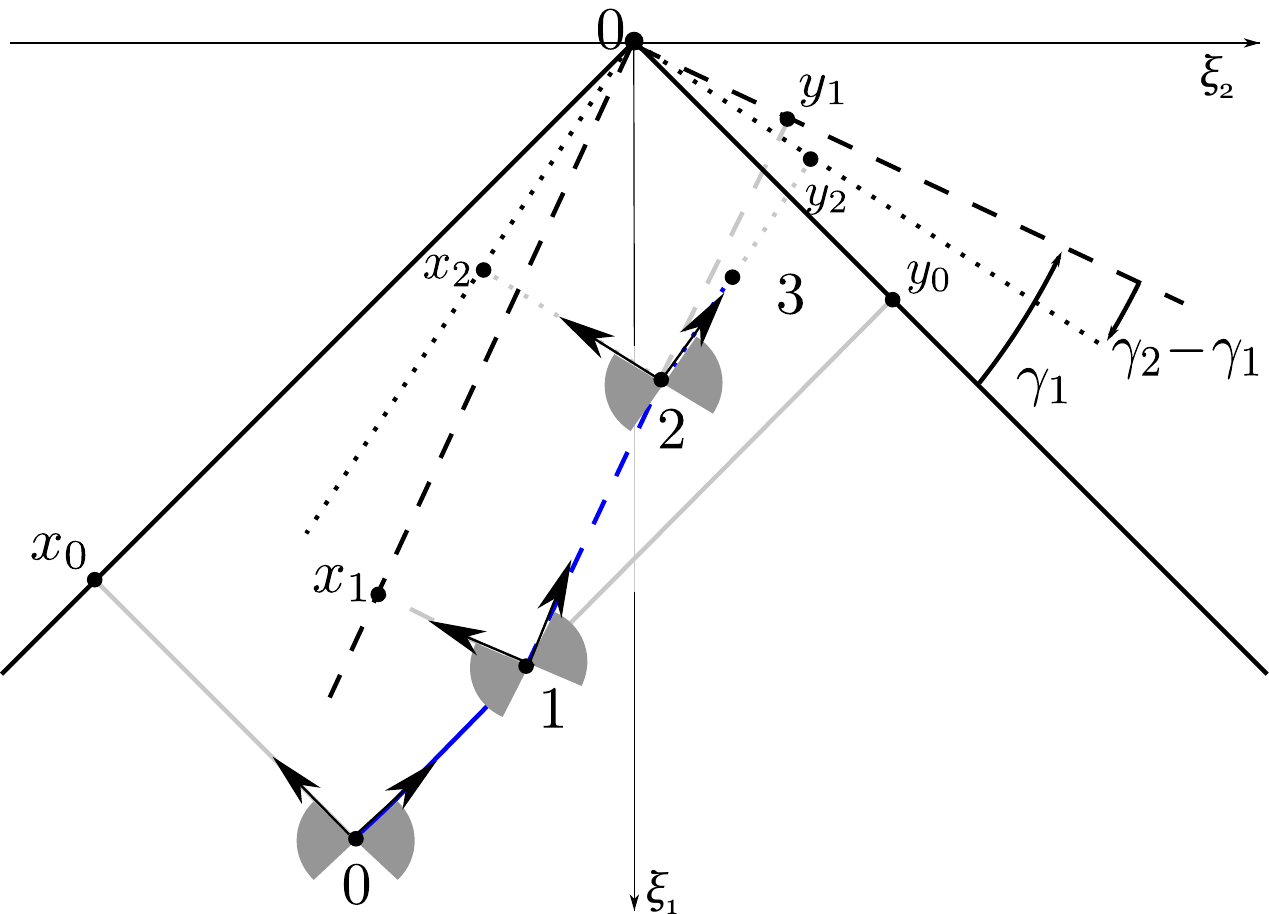}\hspace{1ex}
\includegraphics[height=7cm, width=7cm,
keepaspectratio]{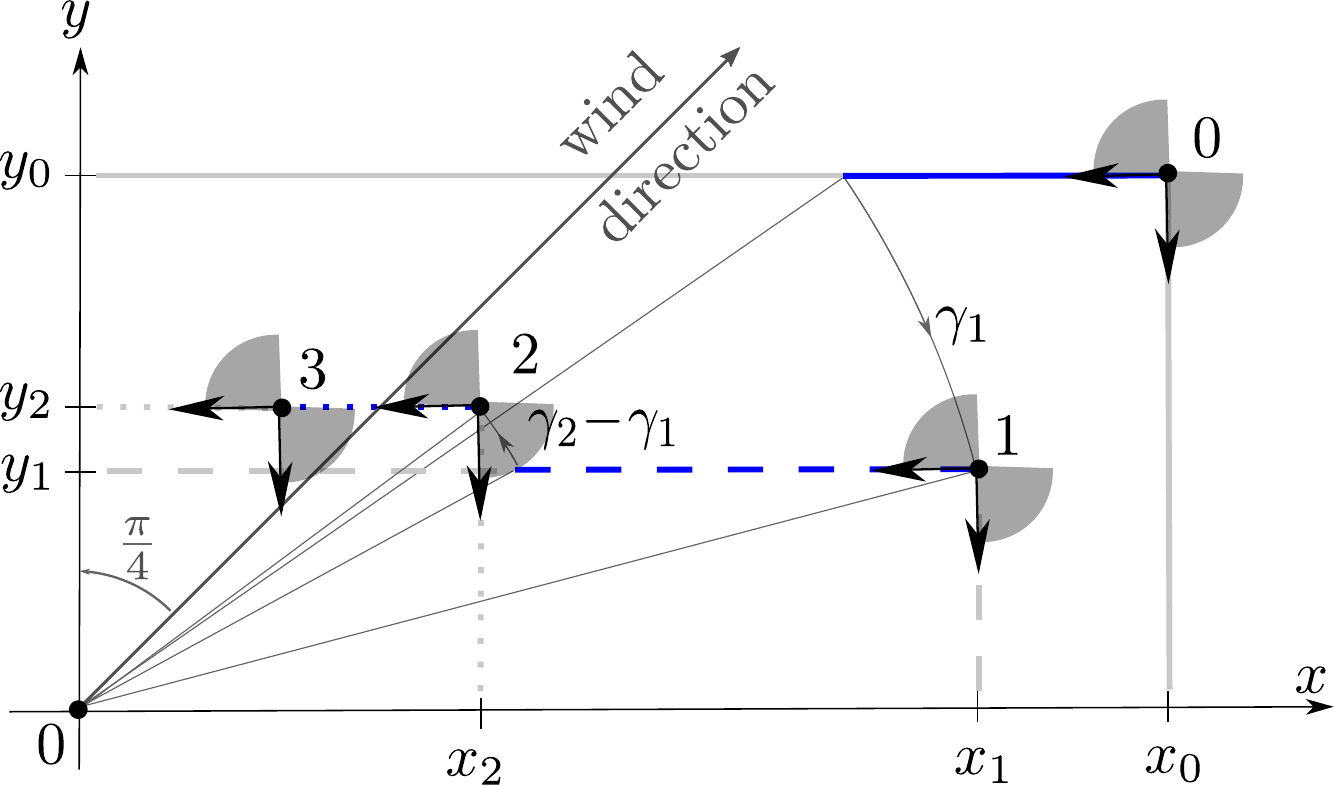}
\caption{The picture on the left-hand side uses the geographic reference frame to show the trajectory of a boat starting at position $0$ and sailing on port tack. The wind direction $\beta_t$ in the geographic frame is given by $\beta_t =  \gamma_1 \, 1_{[\tau_1, \tau_2[}(t) + \gamma_2\, 1_{[\tau_2, \tau_3[}(t)$, where $0 < \tau_1 < \tau_2 < \tau_3$ and $0 < \gamma_2 < \gamma_1 < \piq$. We simplify the notation by using  $x_i := x_{\tau_i}$ and $y_i := y_{\tau_i}$, $i=0, 1, 2$.  The butterfly-shaped regions correspond to the feasible directions when sailing on port tack (resp.~starboard tack). The boat follows the solid line to move from position $0$ to position $1$ between times $0$ and $\tau_1$ while the wind is coming from the North, then, when the wind direction changes to $\beta_{\tau_1} = \gamma_1$, the boat follows the dashed line to move from position $1$ to position $2$ between times $\tau_1$ and $\tau_2$, and when the wind direction changes to $\gamma_2$ (which implies a change of $\gamma_2 - \gamma_1$ from the previous direction), the boat follows the dotted line to move from position $2$ to position $3$ between times $\tau_2$ and $\tau_3$.
The picture on the right-hand side shows the corresponding trajectory of the boat in the rotating reference frame attached to the wind direction. From \cite{CDV2025}.
}
\label{Fig:beating:upwind}
\end{center}
\end{figure}

Following \cite{CDV2025}, we attach the reference frame to the wind
direction rather than to fixed geographic directions. In this rotating
frame, the radial coordinate $r$ denotes the distance from the buoy, while
the angular coordinate $\t = \piq + \zeta$, where $\zeta$  is the angle between the wind
direction and the line of sight connecting the boat to the target. Thus changes in
the wind direction appear as rotations of the boat around the target, see right hand side of Fig. \ref{Fig:beating:upwind}.

Observe the different meaning of the angular coordinate $\t$ with respect to the proportional navigation model where it denotes the lead angle (angle between the pursuer's velocity vector and the line of sight).
%
%In this new reference frame, the motion on starboard (resp. port) tack is shown on the right (resp. left) hand side of Fig.~\ref{Fig:4}.
%
%\begin{figure}
%\begin{center}
%\includegraphics[height=5cm, width=5cm,
%keepaspectratio]{motion_port_tack_simplified_2.pdf}
%\includegraphics[height=5cm, width=5cm,
%keepaspectratio]{motion_starboard_tack_simplified_2.pdf}
%\caption{The sailboat's motion on port tack on the left hand side and on starboard tack on the right hand side. Reference frame is attached to the wind direction. From \cite{CDV1}}.
%\label{Fig:4}
%\end{center}
%\end{figure}

%In \cite{CDV2025}, the reference frame is attached to the buoy and rotates with the wind direction which is always parallel to the vector $[1,1]$ of the new reference frame. Therefore, in this frame, the target is located at the origin and it is the boat that rotates (hence the LOS, the line of sight connecting the boat to the target). Because of the no-go-zone, the boat's angle of direction has to be $>45^ \circ$ or $<-45^\circ$ with the respect to the wind direction ($[1,1]$). Hence, at most, the boat can proceed either vertically or horizontally towards the target (see top right quadrant of Figure \ref{fig_drift}).

In \cite[Sec. 5]{CDV2025} it is considered the  feedback approach that selects, at any given position in the race field, the tack that maximizes the projection of the boat's speed along the radial direction. Thus, picking the best heading at each sailing angle from right hand side of Fig. \ref{Fig:beating:upwind}, we obtain the left hand side of Fig. \ref{fig_drift}.

\begin{figure}[h]
\centering
\includegraphics[height=4cm, width=4cm,keepaspectratio]{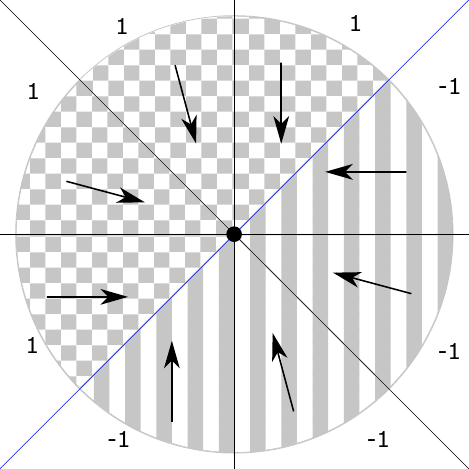}\hspace{1cm}
\includegraphics[height=6cm, width=6cm,keepaspectratio]{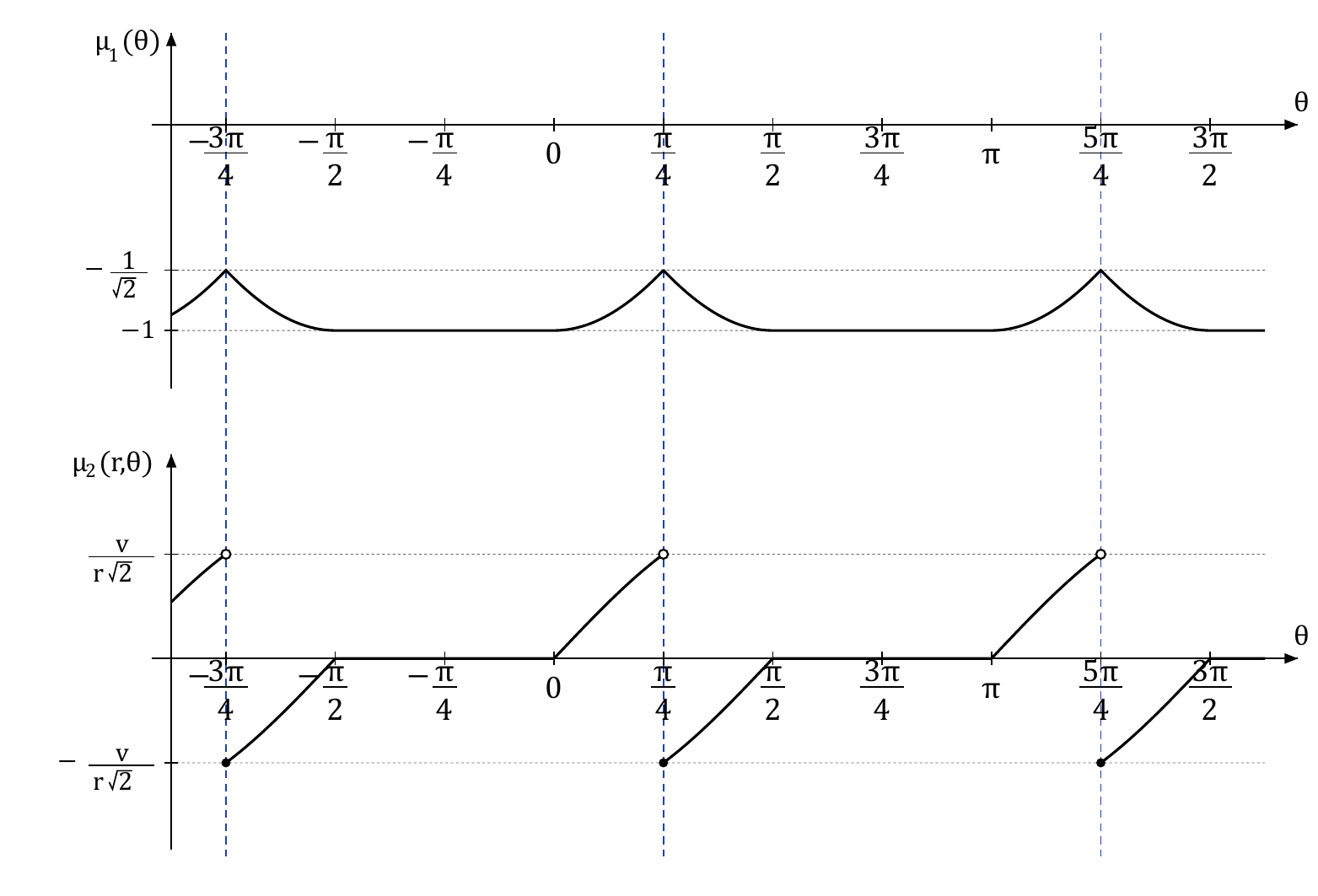}
\caption{Left: direction of the drift under the rotating reference frame under the equation of motion \eqref{state} with coefficients \eqref{eq_sail}. The wind direction is parallel to the vector $[1,1]$. The outer numbers $+1$ (resp. $-1$) indicate whether the wind hits the boat from the right (resp. left).
Right: drift $\mu_1$ and $\mu_2$ extended by $2\pi-$periodicity. 
From \cite{CDV2025}.}
\label{fig_drift}
\end{figure}

The equations of motion in polar coordinates $(R_t, \Theta_t)$ are derived in Section 2 of \cite{CDV2025} and is in the form of a system of stochastic differential equations of exactly the form \eqref{state} with the following coefficients:
\begin{equation}\label{eq_sail}
 \mu_1(\t)= \left\{\begin{array}{l}
-v \cos(\theta),
\\ -v \sin(\theta), 
\\
-v, 
\\
\ \ v\cos(\theta),
\\ v \sin(\theta),
\\ -v,
\end{array}
\right.
\quad \mu_2(r,\t)= \left\{\begin{array}{cl}
 \frac{v \sin(\theta)}{r}, &\qquad \theta \in [0, \frac{\pi}{4} [,\\
- \frac{v \cos(\theta)}{r}, &\qquad \theta \in [\frac{\pi}{4}, \frac{\pi}{2} [,\\
0, &\qquad \theta \in [\frac{\pi}{2}, \pi [,\\
- \frac{v\sin(\theta)}{r}, &\qquad \theta \in [\pi, \frac{5 \pi}{4} [,\\
\frac{v\cos(\theta)}{r}, &\qquad \theta \in [\frac{5 \pi}{4}, \frac{3 \pi}{2} [,\\
0, &\qquad \theta \in [\frac{3 \pi}{2}, 2\pi [,
\end{array}
\right.
\end{equation}

\subsection{Discussion on the models}
While the two models have similar coefficients, there are some important differences.

In \eqref{eq_pn}, the radial drift $\mu_1$ may vanish. In particular, no uniform bound
on the hitting time of the origin is available; we can only prove that it is finite a.s.,
see Proposition~\ref{tau_f}. In Theorem~\ref{Solution_pn}, however, it turns out that
a uniform bound is required only on a favorable event, and this is proved in
Proposition~\ref{tau_f_sw}.

By \emph{critical angles} we mean the attracting equilibria of the angular drift,
i.e.\ the angles toward which $\mu_2$ pushes the process from both sides. Near such
an angle the factor $1/r$ in $\mu_2$  cause the angular drift to blow up as
$r\to 0$. The two models differ
precisely here. In \eqref{eq_pn} the angular drift $\mu_2=\mp\frac{v\sin\theta}{r}$
vanishes \emph{at} its critical angles $\{j\pi : j\in\Z\}$, since $\sin(j\pi)=0$. In \eqref{eq_sail}, by
contrast, $\mu_2$ vanishes on the intervals $[\tfrac\pi2,\pi[$ and
$[\tfrac{3\pi}{2},2\pi[$, whereas its critical angles are $\{\piq + j \pi : j\in\Z\}$,
at which $\mu_2$ does \emph{not} vanish (indeed it is discontinuous there, with
nonzero one-sided values). Because of this difference in critical-angle structure,
the final step, Item~7) in the proofs of Theorem~\ref{Solution_sail} and
Theorem~\ref{Solution_pn}, must be handled differently in the two cases.

\subsection{Existence and uniqueness of a strong solution up to $\eta >0$}
As a preliminary result, we show existence and uniqueness of a strong solution up to reaching a ball of radius $\eta >0$, centered around the origin, by applying \cite[Thm.~2]{Veretennikov1983}.

Define  
\begin{equation}\label{tau_eta_def}
\tau_\eta := \inf\{t:\, R_t \leq \eta \}.
\end{equation}
Observe that $\tau_\eta$ is increasing as $\eta \ra 0$, so there exists $\tau$ such that $\tau_\eta \ra \tau$ a.s. (possibly $\tau= \infty$).
Then, the following stopping time is well-defined:
\begin{equation}\label{tau_a}
\tau := \lim_{n \ra \infty}\tau_\eta.
\end{equation}

\begin{rem}\label{1.7}
 We can characterize $\tau$ as the first time $R_{t}$ hits the origin. Indeed, for $t< \tau$, then there exists $\eta$ such that $t\leq \tau_\eta$ so $R_{t} \geq \eta$ . This implies $R_{t} >0$ for all $t < \tau$. So we can define $\tau$ as:
\begin{align}\label{tau_a_def1}
\tau= \inf\{t: R_{t} = 0 \}.
\end{align}
\end{rem}

\begin{dfn}\label{def:sol_tau}
Let $\tau=\inf\{t\ge 0:R_t=0\}$. We say that an
$(\mathcal F_t)$-adapted process $(R_t,\Theta_t)_{t\in[0,\tau]}$ is a
\emph{strong solution to \eqref{state} up to $\tau$} if:
\begin{enumerate}
\item for every $t<\tau$ one has $R_t>0$, and $(R_s,\Theta_s)_{s\in[0,t]}$ is a
strong solution to \eqref{state} on $[0,t]$ in the usual sense of
\cite[Def.~2.1, Chap.~5]{ks}; in particular it is continuous and the integrability
condition
\begin{equation}\label{eq_int_con}
\int_0^t\bigl(|\mu_1(R_s,\Theta_s)|+|\mu_2(R_s,\Theta_s)|\bigr)\,ds<\infty
\end{equation}
holds for every $t<\tau$, and
\[
R_t=r+\int_0^t\mu_1(R_s,\Theta_s)\,ds,\qquad
\Theta_t=\theta+\int_0^t\mu_2(R_s,\Theta_s)\,ds+\sigma B_t;
\]
\item the limit $(R_\tau,\Theta_\tau):=\lim_{t\nearrow\tau}(R_t,\Theta_t)$ exists
almost surely.
\end{enumerate}
\end{dfn}

\begin{rem}
The restriction to $t<\tau$ in Item 1. is essential: the integrability
condition \eqref{eq_int_con}  fails at the
endpoint, since $R_s\searrow 0$ as $s\nearrow\tau$ forces the angular drift
$\mu_2(r,\theta)=\frac{f(\theta)}{r}$ to be non-integrable on $[0,\tau]$ in general. The solution
is therefore characterized at $\tau$ through the almost sure limit in
Item 2., not through the equation \eqref{state}; in particular we make no
claim about continuation beyond $\tau$. %Since $\mu_1$ is bounded, the radial
%identity nonetheless extends to the endpoint, giving
%$R_\tau=r+\int_0^\tau\mu_1(R_s,\Theta_s)\,ds=0$.
\end{rem}

%We recall that existence and uniqueness results for SDEs with monotone coefficients are established, for instance, in \cite[Sec.~3]{krylov1981} and \cite[Theorem~3.1.1]{liu2015}. However, these results cannot be applied directly as they require the coefficients to be continuous with respect to the space variable. This continuity assumption is not imposed in item. 2 because it would be too restrictive, for example the coefficients arising in the sailboat trajectory optimization problem are not continuous.

%On the other hand, results for SDEs with discontinuous drift, such as \cite{veretennikov1982}, would be applicable only in the nondegenerate diffusion case. This is not the setting considered here. Therefore, our proof relies instead on the classical Yamada--Watanabe argument \cite[Chapter~5, Section~D]{ks}.

\begin{prop}\label{Solution}
Let  $\eta > 0$ and $\mu_1, \mu_2$ be as in \eqref{eq_pn} or \eqref{eq_sail}.
For every $(r,\t)\in \R_+ \times \R$ there exists a unique strong solution to \eqref{state} on $[0,\tau_\eta]$.
\end{prop}

\begin{proof}
The system \eqref{state} is of the form considered in \cite{Veretennikov1983}, with the identification, in his notations,  $x = r$ (degenerate component, no noise) and $y = \t$ (non-degenerate component, driven by $\sigma dB_t$):
\begin{equation}
\begin{cases}
dR_t = \mu_1(\T_t)\, dt, \\
d\T_t = \mu_2(R_t, \T_t)\, dt + \sigma\, dB_t.
\end{cases}
\end{equation}
%Still using the notations of \cite{Veretennikov1983}, this corresponds to $\tau \equiv 0$, $\beta = \mu_1$, $\sigma$ constant, and $b = \mu_2$.

We verify the hypotheses of \cite[Thm.~2]{Veretennikov1983}, taking into account Remark~1 there which weakens the requirement of two bounded derivatives in $r$ to one bounded derivative when the diffusion is identically zero in the $r$-direction (which is the case here).

\noindent \textit{Standing hypothesis  and Remark~1.} The coefficient $ \mu_2$ (resp. $\sigma$) must have one (resp. two) bounded derivative(s) in $r$. Since $\sigma$ is constant, the latter is trivial.  $\mu_2(r, \t)$ is of the form $\frac{f(\t)}{r}$ so it also satisfies the assumption on $[\eta, \infty[$.

The coefficients $\sigma, \mu_1$ must be Lipschitz in $\t$, which is the case.

\noindent\textit{Hypothesis of Theorem~2.} All coefficients must be Lipschitz in $r$, and $ \mu_2$ should just be bounded and measurable in $y = \t$.
\begin{enumerate}
\item $\mu_1$ does not depend on $r$, hence trivially Lipschitz in $r$.
\item $\mu_2$ is Lipschitz in $r$ on $[\eta, \infty[$, as computed above.
\item $\sigma$ constant, trivially Lipschitz.
\item $\mu_2$ is bounded measurable in $\t$: it is piecewise smooth with a finite number of jumps on $[0,2\pi[$, and $|\mu_2(r, \t)| \leq \frac{v}{\eta}$ uniformly.
\end{enumerate}

\noindent\textit{Non-degeneracy on the $y$-component.} The diffusion coefficient acting on $y_t = \T_t$ is the constant $\sigma > 0$, satisfying the uniform non-degeneracy condition.

To conclude, all hypotheses of \cite[Thm.~2, Remark~1]{Veretennikov1983} are satisfied on the domain $\{R \geq \eta\}$. By that theorem, the system admits a  unique strong solution up to the first exit time from this domain, which is exactly $\tau_\eta = \inf\{t \geq 0 : R_t \leq \eta\}$. 
\end{proof}

With the coefficients as in  \eqref{eq_sail}, an alternative proof to Proposition \ref{Solution}, relying on the classical Yamada--Watanabe argument \cite[Chapter~5, Section~D]{ks},   is given in \cite[Theorem 5.3]{CDV2025}.

\begin{rem}
If $\mu_1,\mu_2$ are as in \eqref{eq_sail}, then the following bound holds:
\begin{equation}\label{tau_bound}
 \tau_\eta \leq \sqrt{2}\, \frac{r- \eta}{v}  \qquad a.s.  
\end{equation}

Indeed the process $R_t$ satisfies the equation
\begin{align*}
R_{\tau_\eta} = r + \int_0^{\tau_\eta} \mu_1(R_t,\Theta_t) dt.
\end{align*}
%hence $R$ does hit the ball of radius $\eta$ in a finite amount of time since it is easily seen that, 
Since  $\mu_1(\t) \leq  -\frac{v}{\sqrt{2}}$ for every $\theta \in \R$, it follows that
\begin{align*}
\eta= R_{\tau_\eta} = r + \int_0^{\tau_\eta} \mu_1(R_t, \Theta_t)\, dt \leq r- \frac{v}{\sqrt{2}} \, \tau_\eta,
\end{align*}
and the conclusion follows. 
\end{rem}

\section{Strong solution up to $\tau$ -- Sailboat Navigation}

We start by proving strong well-posedness up to the origin for the sailboat navigation problem. Since the hitting time to the origin is uniformly bounded, the argument requires fewer intricacies than the corresponding one for the proportional navigation problem. 

%Next Theorem is the first main result of the paper, where we prove that \eqref{state} with the coefficients satisfying the Sailboat Navigation problem, \eqref{eq_sail}  admits a unique  strong solution up to the origin.

Before its proof we need a preparatory Lemma.

\begin{lem}\label{eq_h_bc}
Let $(\tau_{\frac 1m})_{m \geq 1}$ be as in \eqref{tau_eta_def}. Fix constants $\kappa > 0$, $v > 0$, $\sigma > 0$, and $\alpha \in \left]0, \tfrac{1}{2}\right[$, and define
\[
H_m := \left\{ \sup_{\tau_{\frac 1m} \leq s \leq \tau_{\frac 1m} + \frac{\kappa}{mv}} \left|B_s - B_{\tau_{\frac 1m}}\right| \leq \frac{m^{-\alpha}}{2\sigma} \right\}, \qquad m \geq 1.
\]
Then
\[
\sum_{m=1}^{\infty} \mathbb{P}(H_m^c) < \infty,
\]
and consequently, by the Borel--Cantelli lemma,
\[
\mathbb{P}\!\left(\liminf_{m \to \infty} H_m\right) = 1.
\]
\end{lem}

\begin{proof}
By the strong Markov property applied at the stopping time $\tau_{\frac 1m}$, the process $s \mapsto B_{\tau_{\frac 1m} + s} - B_{\tau_{\frac 1m}}$ is a standard Brownian motion independent of $\mathcal{F}_{\tau_{\frac 1m}}$. Hence
\[
\mathbb{P}(H_m^c) = \mathbb{P}\!\left(\sup_{0 \leq s \leq \frac{\kappa}{mv}} |B_s| > \frac{m^{-\alpha}}{2\sigma}\right).
\]
By the Brownian scaling identity $\sup_{0 \leq s \leq \frac{\kappa}{mv}} |B_s| \stackrel{d}{=} \sqrt{\tfrac{\kappa}{mv}}\, \sup_{0 \leq u \leq 1} |B_u|$,
\[
\mathbb{P}(H_m^c) = \mathbb{P}\!\left(\sup_{0 \leq u \leq 1} |B_u| > \frac{m^{\frac{1}{2}-\alpha}}{C}\right),
\]
where
\[
C := 2\sigma\sqrt{\tfrac{\kappa}{v}}.
\]
By the law of the supremum of Brownian motion (see e.g.\ \cite{revuzYor}, Remark 8.3 in Section 2.8),
\[
\mathbb{P}\!\left(\sup_{0 \leq u \leq 1} |B_u| > \frac{m^{\frac{1}{2}-\alpha}}{C}\right) \leq 4 \int_{\frac{m^{\frac{1}{2}-\alpha}}{C}}^{\infty} \frac{1}{\sqrt{2\pi}}\, e^{-x^2/2}\, dx.
\]
Since $\alpha < \tfrac{1}{2}$, the threshold $\frac{m^{\frac{1}{2}-\alpha}}{C}$ diverges as $m \to \infty$. For $m$ large enough that $\frac{m^{\frac{1}{2}-\alpha}}{C} \geq 1$,
\[
\int_{\frac{m^{\frac{1}{2}-\alpha}}{C}}^{\infty} \frac{1}{\sqrt{2\pi}}\, e^{-x^2/2}\, dx \;\leq\; \int_{\frac{m^{\frac{1}{2}-\alpha}}{C}}^{\infty} x\, \frac{1}{\sqrt{2\pi}}\, e^{-x^2/2}\, dx \;=\; \frac{1}{\sqrt{2\pi}}\, e^{-\frac{m^{1-2\alpha}}{2C^2}}.
\]
To conclude,
\[
\mathbb{P}(H_m^c) \;\leq\; \frac{4}{\sqrt{2\pi}}\, e^{-\frac{m^{1-2\alpha}}{2C^2}} \qquad \text{for all $m$ large enough.}
\]
Since $1 - 2\alpha > 0$, the right-hand side is summable in $m$, so $\sum_{m=1}^{\infty} \mathbb{P}(H_m^c) < \infty$. The Borel--Cantelli lemma yields $\mathbb{P}(\limsup_m H_m^c) = 0$, equivalently $\mathbb{P}(\liminf_m H_m) = 1$.
\end{proof}

\begin{thm}\label{Solution_sail}
Let $\mu_1,\mu_2$ be as in \eqref{eq_sail} and $\eta = 0$. Then  for every $(r,\t)\in \R_+ \times \R$ there exists a unique strong solution to \eqref{state} in $[0,\tau]$ in the sense of Definition \ref{def:sol_tau}.
\end{thm}

\begin{proof}

We first  prove that the conditions of Definition \ref{def:sol_tau} are satisfied. Item 1. follows from Proposition \ref{Solution}, so it remains to prove that
 for every $(r,\t)\in \R_+ \times \R$
\begin{align}
\lim_{t\nearrow \tau} (R_t,\T_t )
\end{align}
exists a.s.  

Since the process $ R_t$ is bounded from below and decreasing, then it always has a limit, in particular by Remark \ref{1.7}, it follows that $ \lim_{t\ra \tau} R_t = 0 $. It remains to show that $ \lim_{t\ra \tau} \T_t $ exists.

We prove that
\begin{equation}
\lim_{t \nearrow \tau} \Theta_t \in \bigcup_{j \geq 0}\left( \left](2j+1)\tfrac{\pi}{2},\, (2j+2)\tfrac{\pi}{2}\right[ \,\cup\, \left\{ \tfrac{\pi}{4} + j\pi \right\} \right)\quad \text{a.s.}
\end{equation}
This reflects two types of limiting behavior. On the intervals $\left](2j+1)\tfrac{\pi}{2},\, (2j+2)\tfrac{\pi}{2}\right[$ — namely $\left]\tfrac{\pi}{2}, \pi\right[$ and $\left]\tfrac{3\pi}{2}, 2\pi\right[$ modulo $2\pi$ — the angular drift vanishes ($\mu_2 \equiv 0$), and $\Theta_t$ evolves as a pure Brownian motion, which can converge to any point of the interval. At the singular angles $\tfrac{\pi}{4} + j\pi$, the drift $\mu_2$ jumps from $+\tfrac{v}{r\sqrt 2}$ to $-\tfrac{v}{r\sqrt 2}$ as $\theta$ crosses the singularity: on either side the drift points back toward the singular angle, and its magnitude blows up as $R_t \to 0$, forcing $\Theta_t$ to converge to it.

\begin{itemize}
\item[(1)] We first prove that $\T_t$ cannot wander to $\pm \infty$, that is:
\begin{equation}
\liminf_{t \nearrow \tau} \T_t >-\infty, \qquad \limsup_{t \nearrow \tau} \T_t < +\infty.
\end{equation}
\begin{proof}
Otherwise, $\T_t$ would cross infinitely many intervals $\left](2j+1)\tfrac{\pi}{2},\, (2j+2)\tfrac{\pi}{2}\right[$, and each requires a Brownian increment of size $\pm \pim$. This is impossible since $\tau$ a.s. finite (see \eqref{tau_bound}).
\end{proof}

 \hspace*{-\leftmargin}  In the next two items, we prove that $ \limsup_{t \nearrow \tau} - \liminf_{t \nearrow \tau} $ is a.s. uniformly bounded.

\item[(2)] $\forall k\in \Z$, $\forall \ve >0$ the probability that $\Theta_t$ crosses $[\frac{k\pi}{2}-\ve, \frac{k \pi}{2}+\ve]$ infinitely many times is zero.
\begin{proof}
We detail $k=0$; every other level reduces to this case by the
$2\pi$-periodicity of the coefficients together with the reflection
symmetry of \eqref{eq_sail} about the critical angles. In
particular $k$ odd is the mirror image of $k$ even under
$\theta\mapsto-\theta$, and in all cases one side of the level
$\tfrac{k\pi}{2}$ carries zero drift while the other carries a drift
pointing away from the level, which is all the argument uses.

 If $\T_0 = \t_0 =-\ve$, the equations of motion, for $t$ such that $\T_t  \leq 0 $, are
\begin{equation}
\left\{
\begin{array}{l l}
dR_t = -v dt , \\
d\T_t = \s dB_t,
\end{array}
\right. 
\end{equation}
and for $\T_t  \geq 0 $
\begin{equation}
\left\{
\begin{array}{l l}
dR_t = -v\cos \T_t dt ,\\
d\T_t = v \frac{\sin \T_t}{R_t} dt + \s dB_t,
\end{array}
\right. 
\end{equation}
so  crossing from $-\ve$ to $0$ requires a Brownian increment of $+\ve$. For the same reason crossing from $\ve$ to $-\ve$ requires a Brownian increment less than $-2\ve$ since the drift is non negative. Since $\tau < + \infty$ a.s., and $t\ra B_t$ is continuous, there can only be a finite number of such increments in $[0,\tau]$.% indeed for all $\eta  >0$, there is $\d  >0$ such that when $0 <\tau-t<\d$, then $| B_t-B_{\tau}|\leq \eta$, so all increments after time $\tau -\d$ are smaller than $\eta$, then take $\eta < \ve$.
\end{proof}

\item[(3)]
For $\ve >0$ let
 \begin{align}
j_\ve := \sup \left\{ j\in \Z:\, \liminf_{t \nearrow \tau} \T_t \geq \frac{j \pi}{2} -\ve  \right\}
\end{align}
and
\begin{align}
 k_\ve := \inf \left\{ k\in \Z:\, \limsup_{t \nearrow \tau} \T_t \leq \frac{k \pi}{2} +\ve  \right\},
\end{align}
then
\begin{align}
j_\ve = k_\ve \qquad \text{or} \qquad j_\ve +1 = k_{\ve}.
\end{align}
\begin{proof}
Otherwise $j_\ve +2 \leq k_\ve$, so $\liminf_{t \nearrow \tau} \T_t \leq \frac{(j_\ve+1)\pi}{2}-\ve$, $\limsup_{t \nearrow \tau} \T_t \geq \frac{(k_\ve-1)\pi}{2}+\ve$, and $j_\ve+1\leq k_\ve-1$, so the interval $[ \frac{(j_\ve+1)\pi}{2}-\ve, \frac{(k_\ve-1)\pi}{2}+\ve ]$ would be crossed infinitely often, which is impossible by Item (2).
\end{proof}

\item[] 
\begin{adjustwidth}{-\leftmargin}{0pt}
By Item (3), we can claim that 
\begin{equation}\label{big_set1}
\P\left(\bigcup_{j\in \Z} E_{j,\ve}\right) = 1,
\end{equation}
where 
$$E_{j,\ve} =  \left\{ \frac{j \pi}{2}-\ve \leq \liminf_{t \nearrow \tau} \T_t \leq  \limsup_{t \nearrow \tau} \T_t  \leq \frac{(j +1)\pi}{2}+\ve \right\}.$$

Define
\begin{align}
F_{j,\ve} &:= \left\{ \frac{j \pi}{2}-\ve \leq \liminf_{t \nearrow \tau} \T_t ,\,   \frac{j \pi}{2}+ \ve < \limsup_{t \nearrow \tau} \T_t  \leq \frac{(j +1)\pi}{2}+\ve \right\},\\
G_{j,\ve} &:= \left\{ \frac{j \pi}{2}-\ve \leq \liminf_{t \nearrow \tau} \T_t ,\,    \limsup_{t \nearrow \tau} \T_t  \leq \frac{j \pi}{2}+\ve \right\}.
\end{align}
Then for every $\ve>0$:
\begin{align}
F_{j,\ve} \cap G_{j,\ve} = 0,\qquad F_{j,\ve} \cup G_{j,\ve} = E_{j,\ve}, 
\end{align}
and 
\begin{align}
\P\left(\bigcup_{j\in \Z} (F_{j,\ve} \cup G_{j,\ve}) \right) = 1, \label{fcupg}
\end{align}

The next Item 4) and Item 5) further refine \eqref{fcupg} and prove that 
\begin{equation}\label{eq_33a}
\P\Big(\bigcup_{j\in \Z}  (L_j \cup K_j)\Big) = 1,
\end{equation}
where
\begin{align}\label{claim2}
& L_j =  \left\{  \frac{j\pi}{2} < \liminf_{t\nearrow \tau} \T_t \leq \limsup_{t
\nearrow \tau} \T_t < \frac{(j+1)\pi}{2}  \right\} \qquad \text{or}\\
&K_j = \left\{\lim_{t\nearrow \tau} \T_t = \frac{j\pi}{2} \right\}.
\end{align}
In particular, $\liminf_{t\nearrow \tau}$ and $\limsup_{t\nearrow \tau}$ are both within the same quadrant. 

We start by proving a weaker statement than \eqref{eq_33a}.
\end{adjustwidth}

\item[(4)] $\P\Big(\bigcup_{j\in \Z}  (K_j\cup K_j^c) \Big)$ = 1, where
\begin{align}
K_j = \left\{\lim_{t\nearrow \tau} \T_t = \frac{j\pi}{2} \right\},
\end{align}
or
\begin{align}
K_j^c = \left \{ \frac{j\pi}{2} \leq \liminf_{t\nearrow \tau} \T_t < \frac{(j+1)\pi}{2} ; \, \frac{j\pi}{2}  < \limsup_{t\nearrow \tau} \T_t \leq \frac{(j+1)\pi}{2} \right \}.
 \end{align}
Observe that on $K^c_j$ no convergence to either extremity  happens.

\begin{proof}

We first prove the following claim.

\textit{Claim.} With probability $1$, there exists a unique $j \in \Z$ such that either $F_{j,\ve}$ or $G_{j,\ve}$ occurs for all sufficiently small $\ve > 0$.

\begin{proof}[Proof of the Claim]
Write $\ell^- := \liminf_{t \nearrow \tau} \T_t$ and $\ell^+ := \limsup_{t \nearrow \tau} \T_t$. By \eqref{fcupg}, for every $\ve > 0$ there exists $j(\ve) \in \Z$ such that $F_{j(\ve),\ve} \cup G_{j(\ve),\ve}$ occurs a.s.

We split into two cases according to whether $\ell^- = \ell^+ = \frac{j\pi}{2}$ for some $j \in \Z$, or otherwise.

\smallskip

\textit{Case 1: $\ell^- = \ell^+ = \frac{j\pi}{2}$ for some $j \in \Z$.}
Then $G_{j,\ve}$ holds for every $\ve > 0$.

\smallskip

\textit{Case 2: otherwise.}
By \eqref{fcupg}, there exist $j \in \Z$ and $\ve > 0$ such that $F_{j,\ve}$ occurs (indeed, $G_{l,\ve}$ fails for every $l$ once $\ve$ is sufficiently small, so only $F$-events can occur for small $\ve$).

If, for $\ve_1 < \ve$, $F_{l,\ve_1}$ happens for $l \neq j$, then $\ell^+$ would belong to disjoint intervals for $\ve_1$ small enough, a contradiction. Hence, for $\ve_1 < \ve$, at most $F_{j,\ve_1}$ can happen. It remains to prove that $F_{j,\ve_1}$ does happen for all sufficiently small $\ve_1 < \ve$.

For every $k \in \Z$, it cannot happen that $\ell^+ > \frac{k\pi}{2}$ and $\ell^- < \frac{k\pi}{2}$ simultaneously. Indeed, this would mean $\T_t$ crosses $\frac{k\pi}{2}$ with amplitude at least $\delta := \min\!\left(\ell^+ - \frac{k\pi}{2},\ \frac{k\pi}{2} - \ell^-\right) > 0$ infinitely often as $t \nearrow \tau$. By the model, the drift pushes $\T$ away from $\frac{k\pi}{2}$ on at least one side; on that side, every excursion of $\T$ back toward $\frac{k\pi}{2}$ requires the Brownian motion to supply a fluctuation of size $\ge \delta$ against the drift. Since $\tau < \infty$ a.s., the Brownian motion has only finitely many such fluctuations on $[0,\tau]$, a contradiction.

Hence there exists $j \in \Z$ for which $\frac{j\pi}{2} \leq \ell^- <  \frac{(j+1)\pi}{2}$ and $\frac{j\pi}{2} < \ell^+ \leq \frac{(j+1)\pi}{2}$, that is, $F_{j,\ve}$ occurs for all sufficiently small $\ve > 0$.
\end{proof}

\smallskip

In both cases, there is a unique $j \in \Z$ such that $F_{j,\ve} $ or $G_{j,\ve}$ occurs for all sufficiently small $\ve > 0$, which is equivalent to saying that with probability 1 there exists $j$ such that either $K_j$ or $K^c_j$ occurs.
\end{proof}

 \hspace*{-\leftmargin}\parbox{\dimexpr\linewidth+\leftmargin\relax}{ Building on Item 4), we now establish the stronger statement that  $K^c_j = L_j$, that is the pair $(\liminf_{t \nearrow \tau} \T_t,\, \limsup_{t \nearrow \tau} \T_t)$ cannot have one coordinate in the interior of the interval and the other at the boundary. This proves  \eqref{eq_33a}}

\item[(5)] For $j \in \Z$, 
\begin{equation}\label{1.9.6.1}
\P\left\{ \liminf_{t\nearrow \tau} \T_t = \frac{j\pi}{2}, \frac{j\pi}{2}  <\limsup_{t\nearrow \tau} \T_t \leq  \frac{(j+1)\pi}{2}   \right\}   = 0,
\end{equation}
and
\begin{equation}\label{1.9.6.2}
 \P \left\{ \frac{j\pi}{2} \leq \liminf_{t\nearrow \tau} \T_t < \frac{(j+1)\pi}{2}, \limsup_{t\nearrow \tau} \T_t = \frac{(j+1)\pi}{2}  \right \} = 0.
\end{equation}
In particular, $K^c_j =  L_j$ and \eqref{eq_33a} holds.
\begin{proof}

Suppose \eqref{1.9.6.1} occurs for some $j$
(\eqref{1.9.6.2} is symmetric and is treated at the end). Write
\[
\underline{\Theta} := \liminf_{t\nearrow\tau}\Theta_t = \frac{j\pi}{2},
\qquad
\overline{\Theta} := \limsup_{t\nearrow\tau}\Theta_t \in
\Bigl]\tfrac{j\pi}{2},\,\tfrac{(j+1)\pi}{2}\Bigr],
\]
and pick $\varepsilon$ with $0<\varepsilon<\overline{\Theta}-\tfrac{j\pi}{2}$.
Then $\Theta_t$ enters $[\tfrac{j\pi}{2},\,\tfrac{j\pi}{2}+\varepsilon]$
infinitely often (approaching $\underline{\Theta}$) and exceeds
$\tfrac{j\pi}{2}+\varepsilon$ infinitely often (approaching $\overline{\Theta}$),
so $\Theta$ crosses the band
$[\tfrac{j\pi}{2},\,\tfrac{j\pi}{2}+\varepsilon]$ downward infinitely
often as $t\nearrow\tau$.

On $\theta\in[\tfrac{j\pi}{2},\,\tfrac{j\pi}{2}+\varepsilon]$ the angular
drift never points toward $\tfrac{j\pi}{2}$: by \eqref{eq_sail},
for $j$ even it is nonnegative there, and for $j$ odd it vanishes
($\mu_2\equiv 0$). Writing a downcrossing from
$\tfrac{j\pi}{2}+\varepsilon$ to $\tfrac{j\pi}{2}$ over a time interval
$[t_1,t_2]$ as
\[
\Theta_{t_2}-\Theta_{t_1}
= \int_{t_1}^{t_2}\mu_2(R_s,\Theta_s)\,ds
+ \sigma\bigl(B_{t_2}-B_{t_1}\bigr),
\]
the integral term is $\ge 0$, hence
\[
\sigma\bigl(B_{t_2}-B_{t_1}\bigr)
\le \Theta_{t_2}-\Theta_{t_1} = -\varepsilon,
\qquad\text{i.e.}\qquad
B_{t_2}-B_{t_1}\le -\frac{\varepsilon}{\sigma}.
\]
Since $\tau<\infty$ a.s.\ and $B$ is uniformly continuous on the compact
interval $[0,\tau]$, only finitely many disjoint increments of magnitude
$\ge \varepsilon/\sigma$ occur on $[0,\tau]$. This contradicts the
infinitely many downcrossings and proves the first identity. 

The second event  \eqref{1.9.6.2}
follows by the mirror argument at the upper boundary
$\tfrac{(j+1)\pi}{2}$, where on the relevant side $\mu_2\le 0$ for $j$
even and $\mu_2\equiv 0$ for $j$ odd. In particular
$K^c_j=L_j$ and \eqref{eq_33a} holds.
\end{proof}

 \hspace*{-\leftmargin}\parbox{\dimexpr\linewidth+\leftmargin\relax}{ Looking back at \eqref{eq_33a}, if there exists $j$ for which
 $K_j$ occurs, then the conclusion follows and the limit exists,
otherwise $L_j$ occurs. }

\item[(6)] If $L_j$ occurs and $j$ is odd, then $\lim_{t\nearrow \tau} \T_t$ exists.
\begin{proof}
On $L_j$ with $j$ odd, $\Theta_t$ is ultimately confined to the interval
$\bigl]\tfrac{j\pi}{2},\,\tfrac{(j+1)\pi}{2}\bigr[$, on which
$\mu_2\equiv 0$; there
\[
  \Theta_t = \Theta_{t_0} + \sigma\bigl(B_t-B_{t_0}\bigr).
\]
Since $B$ is continuous and $\tau<\infty$, the limit
$\lim_{t\nearrow\tau}\Theta_t$ exists.
%Suppose \eqref{1.9.6.1} occurs, and, for instance, $j$ is even, say $j=0$. Then there exists $\ve>0$ for which the interval $[0,\ve]$ would be crossed downwards infinitely many times. Because the drift of $\T$ is positive on this interval, this requires a Brownian increment smaller or equal than $-\ve$. Since this can only happen finitely many times and we do not have $\lim_{t\nearrow \tau} \T_t = 0$,  we get a contradiction. 
%If $j$ is odd, the drift of $\T$ is zero on this interval, and the same argument applies.
\end{proof}

\item[(7)] If $L_j$ occurs and $j$ is even, then $\lim_{t\nearrow \tau} \T_t = \frac{j\pi}{2} + \piq$.
\begin{proof}

By the symmetries of the coefficients, without loss of generality,  it is enough to consider the case $j=0$.
For $m\in \N^\star$, set
\begin{align}\label{eq_hm}
H_m= \left\{\sup_{\tau_{\frac 1m} \leq t \leq \tau_{\frac 1m} + \frac{  \sqrt{2}}{mv} } |B_t-B_{\tau_{\frac 1m}} | \leq \frac{m^{-\frac{1}{4}}}{2\sigma} \right\},
 \end{align}
where  $\tau_{\frac 1m} =\inf  \{t: R_t\leq \frac{1}{m} \}$. Recall that, by Proposition \ref{Solution}, $\tau \leq \tau_{\frac 1m} +  \frac{ \sqrt{2}}{mv }$.
Let
\begin{align}
L_{0,\ve} = \left\{  \ve < \liminf_{t\nearrow \tau} \T_t  \leq \limsup_{t\nearrow \tau} \T_t< \pim -\ve \right\},
 \end{align}
then $L_0 = \cup_\ve L_{0,\ve}$, where $L_0$ is defined in \eqref{claim2}.
Define also
 \begin{align}
 L_{0,\ve,m} = L_{0,\ve} \cap  J_{0,m},
 \end{align}
with 
\begin{equation}\label{eq_jm}
J_{0,m}= \left\{     \frac{m^{-\frac{1}{4}}}{2} <\T_{\tau_{\frac 1m}} <  \pim -  \frac{ m^{-\frac{1}{4}}}{2}  \right\}
\end{equation}

\smallskip
\textit{Claim.}
$L_{0,\ve} = \bigcup_{n\in\N^\star} \bigcap_{m\geq n} L_{0,\ve,m}$.

\begin{proof}[Proof of the claim]

The inclusion $\supseteq$ is immediate from $L_{0,\ve,m} \subseteq L_{0,\ve}$ for every $m$.

For the converse, fix $\omega \in L_{0,\ve}$, so that
$$\ve < \liminf_{t\nearrow\tau}\T_t(\omega) \leq \limsup_{t\nearrow\tau}\T_t(\omega) < \tfrac{\pi}{2} - \ve.$$
By definition of $\liminf$ and $\limsup$ as $t \nearrow \tau$, there exists $\delta > 0$ such that
$$\ve < \T_t(\omega) < \tfrac{\pi}{2} - \ve \qquad \text{for all } t \in ]\tau - \delta, \tau].$$
Since $\tau_{\frac 1m} \nearrow \tau$ as $m \to \infty$, there exists $n_1 \in \N^\star$ such that $\tau_{\frac 1m} \in ]\tau - \delta, \tau]$ for all $m \geq n_1$, and hence
$$\ve < \T_{\tau_{\frac 1m} }(\omega) < \tfrac{\pi}{2} - \ve \qquad \text{for all } m \geq n_1.$$
Choose $n_2 \in \N^\star$ such that $\tfrac{n_2^{-\frac 14}}{2} < \ve$; then for all $m \geq n_2$,
$$\tfrac{m^{-\frac 14}}{2} \leq \tfrac{n_2^{-\frac 14}}{2} < \ve.$$
Setting $n := \max(n_1, n_2)$, we obtain for all $m \geq n$
$$\frac{m^{-\frac 14}}{2} < \ve < \T_{\tau_{\frac 1m }}(\omega) < \frac{\pi}{2} - \ve < \frac{\pi}{2} - \frac{m^{-\frac 14}}{2},$$
i.e.\ $\omega \in L_{0,\ve,m}$ for every $m \geq n$. Hence $\omega \in  \bigcup_{n\in\N^\star}\bigcap_{m\geq n}L_{0,\ve,m}$.
\end{proof}
Therefore, 
\begin{align}
L_0 = \bigcup_{\ve >0} \bigcup_{n\in \N^\star} \bigcap_{m\geq n} L_{0,\ve,m}.
 \end{align}

Assume for the moment the two following facts:
\begin{enumerate}
\item Fact 1. $\sum_{m=1}^\infty \P(H_m^c) < \infty$, so for  $m$ large enough, $H_m$ occurs.
 \item   Fact 2. On $H_m \cap L_{0,\ve,m}$, $ \limsup_{t \nearrow \tau} \sup_{s\geq t} |\T_s-\piq| \leq m^{-\frac{1}{4}}$.
\end{enumerate}

 On $L_0$, there is $\ve >0$ for which $L_{0,\ve}$ occurs. Then for all large $m$, $H_m \cap L_{0,\ve,m}$ occurs, so by Fact 2
 \begin{align}
 \lim \sup_{t \nearrow \tau} \sup_{s\geq t} |\T_s-\piq| \leq  m^{-\frac{1}{4}} \end{align}
for all large $m$, therefore
 \begin{align}
\lim \sup_{t \nearrow \tau} \sup_{s\geq t} |\T_s-\piq| \leq 0,
 \end{align}
and hence it is equal to zero. This is equivalent to $\lim_{t\nearrow \tau} |\T_s-\piq| = 0,$ i.e., $\lim_{t\nearrow \tau} \T_s = \piq $.
\end{proof}
 It remains to prove Fact 1 and Fact 2.

{\em Proof of Fact 1.} 
Follows from Lemma \ref{eq_h_bc}, with  $\kappa = \sqrt{2}$, $\alpha = \tfrac{1}{4}$. %Then $C = 2  \s \sqrt{\sqrt{2}/v} =\s\,   2^{5/4}/\sqrt{v}$.

\smallskip
{\em Proof of Fact 2.}
It follows from Proposition \ref{fact_2}

\item[(8)] It remains to prove pathwise uniqueness of the solution.

 Let $(R,\Theta)$ and
$(\widetilde R,\widetilde\Theta)$ be two strong solutions up to $\tau$ in the
sense of Definition~\ref{def:sol_tau}, driven by the same Brownian motion $B$ and
with the same initial condition $(r,\theta)$. Fix $\eta>0$. On $[0,\tau_\eta]$ the
coefficients satisfy the hypotheses of Proposition~\ref{Solution}, which provides
a \emph{unique} strong solution on $[0,\tau_\eta]$; hence
\[
(R_t,\Theta_t)=(\widetilde R_t,\widetilde\Theta_t)
\qquad\text{for all } t\in[0,\tau_\eta],\ \text{a.s.}
\]
Since $\tau_\eta$ is the exit time from $\{R>\eta\}$ and the two radial paths
coincide on $[0,\tau_\eta]$, the solutions reach the level $\eta$ at the same
instant, so $\tau_\eta=\widetilde\tau_\eta$ for every $\eta>0$. Letting
$\eta\downarrow 0$ and using $\tau_\eta\nearrow\tau$ a.s., we obtain $\tau=\widetilde\tau$ and
\[
(R_t,\Theta_t)=(\widetilde R_t,\widetilde\Theta_t)
\qquad\text{for all } t\in\bigcup_{\eta>0}[0,\tau_\eta]=[0,\tau[,\ \text{a.s.}
\]
Finally, both processes are solutions in the sense of
Definition~\ref{def:sol_tau}, so by Item 2. their endpoint values are the
almost sure limits
\[
(R_\tau,\Theta_\tau)=\lim_{t\nearrow\tau}(R_t,\Theta_t),\qquad
(\widetilde R_\tau,\widetilde\Theta_\tau)
=\lim_{t\nearrow\tau}(\widetilde R_t,\widetilde\Theta_t),
\]
both of which exist a.s. Since the two processes coincide on $[0,\tau[$, these
limits agree, whence
$(R_\tau,\Theta_\tau)=(\widetilde R_\tau,\widetilde\Theta_\tau)$ a.s. This proves
pathwise uniqueness up to $\tau$.
\end{itemize}
\end{proof}

\begin{rem}\label{rem_gamma}
The role of the explicit drift  $\mu_2$ in the proof is confined to
Proposition~\ref{fact_2} (resp.\ Proposition~\ref{tau_f_sw}), through two
features: (i) near each attracting critical angle the drift is \emph{restoring},
i.e.\ it pushes $\Theta$ back toward that angle from both sides; and (ii) the
radial factor $1/r$ produces, after the change of variables $s=\tau_\rho$, a
kernel $1/\rho$ whose integral $\int_q^{1/m}\rho^{-1}\,d\rho=\log\frac{1}{mq}$
diverges as $q\downarrow 0$, dominating the bounded Brownian increment. Any drift
$\mu_2(r,\theta)=f(\theta)/r$ sharing the sign structure in (i) is therefore
handled by the same argument with no change. The case $\mu_2=f(\theta)/r^\gamma$
with $\gamma>1$ is also covered: the kernel becomes $\rho^{-\gamma}$, whose
integral diverges even faster so the restoring estimate only strengthens, although the explicit bounds in
\eqref{eq_gronwall} are modified accordingly. We do not pursue the most general
statement here.
\end{rem}

\begin{prop}\label{fact_2}
Let $\mu_1,\mu_2$ as in \eqref{eq_sail}. 
Then on
$H_m \cap L_{0,\ve, m} $,
\begin{equation}\label{lim_sup0}
\limsup_{t\nearrow \tau}  \sup_{s\geq t} |\T_s - \piq|  \leq m^{- \frac 14}.
\end{equation}

\end{prop}

\begin{proof}
We first prove that $\Theta_s$ hits the level $\piq$ almost surely before time $\tau$ and that, from that hitting time onward, the excursions of $\Theta_s$ around $\piq$, for sufficiently small $R_s$, have amplitude at most $m^{- \frac 14}$.  Choose $q<\frac{1}{m}$ and recall $\tau_{q} = \inf \{t> 0: \, R_t\leq q \}$, so $R_{\tau_{\frac{1}{m}}} > R_{\tau_q} = q$. 
%Let $r = \frac{1}{m}$ such that $R_{\tau_{\frac 1m}} = r$ and $\tau_{\frac 1m} = \tau_{r}$. 
 Then, by \eqref{state},
\begin{equation}\label{336}
\T_{\tau_q} = \T_{\tau_{\frac{1}{m}}} + \int_{\tau_{\frac{1}{m}}}^{\tau_q} \mu_2(R_s,\T_s) ds + \s (B_{\tau_q} - B_{\tau_{\frac{1}{m}}}).
\end{equation}
Observe that $\tau_q$ is differentiable in $q$, indeed $R_{\tau_q} = q$ so $\frac{d R_{\tau_q}}{d q}$ is well defined and by the Leibniz rule
\begin{equation}
\frac{d R_{\tau_q}}{d q} = \frac{d R_{\tau_q}}{d \tau_q} \frac{d\tau_q}{dq} = 1.
\end{equation}
This entails that, since 
\begin{align}
v \frac{\sqrt{2}}{2} \leq  -\frac{d R_t}{dt} \leq v,
\end{align}
then 
\begin{align}\label{0.01}
\frac{1}{v}\leq -\frac{d \tau_q}{dq}\leq \frac{\sqrt{2}}{v}.
\end{align}
Define $\tau^\prime_q:= \frac{d \tau_q}{dq}$.
In the integral in \eqref{336}, we do the change of variable $s= \tau_\rho$:
\begin{align}
\T_{\tau_q} &= \T_{\tau_\frac{1}{m}} + \int_{\frac{1}{m}}^q \mu_2(R_{\tau_\rho},\T_{\tau_\rho}) \tau^\prime_\rho d\rho+ \s (B_{\tau_q} - B_{\tau_\frac{1}{m}})\\
&\qquad = \T_{\tau_\frac{1}{m}} + \int_{q}^{\frac{1}{m}} \mu_2(\rho,\T_{\tau_\rho}) (-\tau^\prime_\rho) d\rho+ \s (B_{\tau_q} - B_{\tau_\frac{1}{m}}).
\end{align}

\smallskip 
\textit{Step 1. $\T_t$ hits $\piq$ a.s. before $\tau$.}
Since  $H_m\cap L_{0,\ve,m}$ occurs, suppose that $\T_{\tau_\frac{1}{m}} \in ]\frac{m^{-\frac 14}}{2}, \piq[$ as the other case $\T_{\tau_\frac{1}{m}} \in ]\piq, \pim-\frac{m^{-\frac 14}}{2}[$ is symmetric and $\T_{\tau_\frac{1}{m}}  = \piq$ is covered in \textit{Step 2}.  
 %Set $\psi_\rho = \T_{\tau_\rho}$. 
 As long as the process is smaller than $\piq$, the SDE \eqref{336} is
\begin{align}
\T_{\tau_q} &= \T_{\tau_\frac{1}{m}} + \int_q^{\frac{1}{m}} \frac{v \sin(\T_{\tau_\rho})}{\rho} (-\tau^\prime_\rho) d\rho+ \s (B_{\tau_q} - B_{\tau_\frac{1}{m}})\\
&\qquad \geq \T_{\tau_\frac{1}{m}} -  m^{-\frac{1}{4}} +  \int_q^{\frac{1}{m}} \sin(\T_{\tau_\rho}) \frac{d \rho}{\rho} \ra + \infty
 \end{align}
as $q \searrow 0$, where we used  \eqref{0.01} and, since $L_{0,\ve,m} \cap H_m$ occurs,  $\sin(\T_{\tau_\rho}) \geq \sin (\ve)$ by possibly choosing a smaller $\frac{1}{m}$.
So necessarily, $\piq$ is hit for the first time at level $r_1\in ]0,\frac{1}{m}[$.

%Now, for $q<r_1$ (recall, $q, r_1$ represent the distance from the origin), since $\mu_2$ is pushing $\theta$ towards the level $\piq$, 
%for $|\T_{\tau_q} -\piq|$ to be larger than $ \s m^{-\frac{1}{4}}$, it requires a Brownian increment larger than  $ \s m^{-\frac{1}{4}}$, which is impossible since $H_m\cap L_{0,\ve,m}$.

\smallskip 
\textit{Step 2. Excursions around $\piq$ have amplitude at most $m^{-\frac{1}{4}}$.}
Let $\xi_q:= \inf \{  l \in [q ,r_1]: \, \T_{\tau_{l}}  = \piq \}$. Fix a $q$ for which $\xi_q >q$, and assume $\T_{\tau_{q}} > \piq$. By construction, also $\T_{\tau_{\rho}} > \piq$ for $\rho \in [q,\xi_q[$. Then the equation of motion becomes
\begin{align}
\T_{\tau_q} &= \piq + \int_q^{\xi_q} \frac{ -v\cos(\T_{\tau_\rho})}{\rho} (-\tau^\prime_\rho) d\rho+ \s (B_{\tau_q}- B_{\tau_{\xi_q}}) \\
&\qquad \leq \piq +m^{-\frac{1}{4}} -  \int_q^{\xi_q} \cos(\T_{\tau_\rho}) \frac{d \rho}{\rho}\\
&\qquad  \leq \piq + m^{-\frac{1}{4}} -  \int_q^{\xi_q} \sin(\ve) \frac{d \rho}{\rho}\\
&\qquad  \ra -\infty
\end{align}
as $q \searrow 0$, where the first inequality follows from  $ -\tau^\prime_\rho \geq \frac{1}{v}$, and the second from $\cos(\T_{\tau_\rho}) \geq \cos (\pim- \ve)$. So, necessarily, there exists a level $\pi_q < q$ for which $\T_{\tau_{\pi_q}} = \piq$ and for any $ \pi_q\leq u\leq \xi_q$, then $ \T_{\tau_u} \leq \piq + m^{-\frac{1}{4}} $.

\smallskip 
\textit{Conclusion.}
Since $q  < r_1$ is generic, let $C= \{s:\, \T_{\tau_s} = \piq \}$; if $q\in C$, then $\xi_q = q$, otherwise if $q\notin C$ and $ \T_{\tau_q} > \piq$ we have proved that there exists a level $\pi_q$ such that $0 <\pi_q<q <\xi_q$ for which $\T_{\tau_{\pi_q}} = \piq$, and for any $ \pi_q\leq u\leq \xi_q$, then $ \T_{\tau_u} \leq \piq +  m^{-\frac{1}{4}} $.
 If $\T_{\tau_q} < \piq$ a similar argument holds with the reverse inequality, that is for any $ \pi_q\leq u\leq \xi_q$, then $ \T_{\tau_u} \geq \piq - m^{-\frac{1}{4}} $.

In particular we have proved that for any $ 0 \leq q \leq r_1 $,
\begin{align}
|\T_{\tau_q} - \piq| & \leq m^{-\frac{1}{4}} ,
\end{align}
This implies that
\begin{align}
\limsup_{t\nearrow \tau}  \sup_{s\geq t} |\T_s-\piq| \leq  m^{-\frac{1}{4}}.
\end{align}
\end{proof}

\section{Strong solution up to $\tau$ -- Proportional Navigation}

Before proving the second and last main result of the work, we need a preliminary proposition because the radial drift $\mu_1$ of the Proportional Navigation model \eqref{eq_pn} may be null and therefore a uniform bound on the time to reach the origin  $\tau$, as in \eqref{tau_bound}, is not available.

Nonetheless, we can prove that $\tau$ is finite a.s. This is a weaker result than the corresponding one for the Sailboat model but is enough to apply the same arguments of Items 1) --5) in the Proof of Theorem \ref{Solution_sail}.

\begin{prop}\label{tau_f}
Let the SDE \eqref{state} with coefficients \eqref{eq_pn}. Then, the hitting time $\tau := \inf\{t \geq 0 : R_t = 0\}$ is finite almost surely.
\end{prop}

\begin{proof}
By the sign convention in \eqref{eq_pn} $\mu_1(\T) = -v|\cos\T|$, so
\begin{equation}\label{eq:radial}
R_t = R_0 - v\int_0^t|\cos\T_s|\,ds \geq 0,
\end{equation}
and $t \mapsto R_t$ is monotone nonincreasing. Suppose for contradiction that $\P(\tau = \infty) > 0$. On $\{\tau = \infty\}$, $R_t > 0$ for all $t$, so $R_t \searrow R_\infty \in [0, R_0]$ as $t \to \infty$, and
\begin{equation}\label{eq:I_finite}
I_\infty := \int_0^\infty|\cos\T_s|\,ds = \frac{R_0 - R_\infty}{v} < \infty.
\end{equation}

We apply Peskir's change-of-variable formula with local time on curves
\cite{Peskir2005} to $g(\Theta_t)$, where $g(\theta):=|\cos\theta|$. The function
$g$ is continuous, and is $C^2$ on each side of the curves
$\theta=\tfrac{\pi}{2}+k\pi$ ($k\in\mathbb Z$), where it has a corner; on either
side
\[
g'(\theta)=-\operatorname{sgn}(\cos\theta)\,\sin\theta,
\qquad
g''(\theta)=-|\cos\theta|.
\]
The formula applies to such a piecewise-$C^{1,2}$ function and to the continuous
semimartingale $\Theta$, with no regularity required of the coefficients of
$\Theta$; in particular the discontinuity of $\mu_2$ at $\pm\tfrac{\pi}{2}$ is
immaterial. On $\{\tau=\infty\}$, the radius $R$ is non-increasing (since
$ dR_s=-v|\cos\Theta_s|\le 0$) and stays strictly positive, so for $s\le t$
one has $R_s\ge R_t>0$; hence the angular drift is bounded,
$|\mu_2(R_s,\Theta_s)|\le v/R_t$ for $s\le t$, and
$\int_0^t|\mu_2(R_s,\Theta_s)|\,ds\le vt/R_t<\infty$. Thus $\Theta$ is a
continuous semimartingale on each $[0,t]$, and the change-of-variable formula
below applies.

 It gives
\begin{align}
g(\Theta_t)-g(\Theta_0)
&=\int_0^t \tfrac12\bigl(g'(\Theta_s+)+g'(\Theta_s-)\bigr)\,d\Theta_s\\
&+\frac12\int_0^t g''(\Theta_s)\, \mathbf 1_{\{ \T_s \notin \{\pim + k\pi \}_{k\in \Z}\}}\,d\langle\Theta\rangle_s\\
&+\frac12\sum_{k\in\mathbb Z}\int_0^t
\bigl(g'(\Theta_s+)-g'(\Theta_s-)\bigr)\,dL_s^{\,\frac{\pi}{2}+k\pi}(\Theta),
\end{align}
where $L_s^{b}(\Theta)$ denotes the local time of $\Theta$ at the level $b$.

We evaluate the three terms. In the first integral, the symmetric average
$\tfrac12\bigl(g'(\Theta_s+)+g'(\Theta_s-)\bigr)$ coincides with $g'(\Theta_s)$
for a.e.\ $s$, since $\Theta$ spends zero time on the curves
$\{\theta=\tfrac{\pi}{2}+k\pi\}$ (a continuous semimartingale with nondegenerate
martingale part occupies any fixed level for zero time). Hence the first integral
is the ordinary It\^o integral of $g'(\Theta_s)$ against
$d\Theta_s=\mu_2(R_s,\Theta_s)\,ds+\sigma\,dB_s$. With $N=2$ the angular drift is
$\mu_2(r,\theta)=-\operatorname{sgn}(\cos\theta)\,\tfrac{v\sin\theta}{r}$, and the
two sign factors cancel:
\[
g'(\Theta_s)\,\mu_2(R_s,\Theta_s)
=\bigl(-\operatorname{sgn}(\cos\Theta_s)\sin\Theta_s\bigr)
 \Bigl(-\operatorname{sgn}(\cos\Theta_s)\tfrac{v\sin\Theta_s}{R_s}\Bigr)
=\frac{v\sin^2\Theta_s}{R_s},
\]
where we used $\operatorname{sgn}(\cos\Theta_s)^2=1$ for a.e.\ $s$ \ (the
exceptional set $\{\cos\Theta_s=0\}$ has zero Lebesgue measure). The
martingale part of the first integral is $-M_t$, with
$M_t:=\sigma\int_0^t\operatorname{sgn}(\cos\Theta_s)\sin\Theta_s\,dB_s$ and
$\langle M\rangle_t=\sigma^2\int_0^t\sin^2\Theta_s\,ds$.

Concerning the second term, since $g''(\theta)=-|\cos\theta|$ off the
curves and $d\langle\Theta\rangle_s=\sigma^2\,ds$,
\[
\frac12\int_0^t g''(\Theta_s)\, \mathbf 1_{\{ \T_s \notin \{\pim + k\pi \}_{k\in \Z}\}}\,d\langle\Theta\rangle_s
=-\frac{\sigma^2}{2}\int_0^t |\cos\Theta_s|\,ds.
\]

The third term is the local times contribution. At each curve
$b_k=\tfrac{\pi}{2}+k\pi$ the derivative jumps upward,
\[
g'(b_k+)-g'(b_k-)=(+1)-(-1)=2,
\]
so, setting $L_t:=\sum_{k\in\mathbb Z}L_t^{\,\frac{\pi}{2}+k\pi}(\Theta)$, this term
equals $L_t$ and is manifestly nonnegative: a sum of (nonnegative) local times
weighted by the (positive) upward jumps of $g'$. Collecting the three terms,
\begin{equation}\label{eq:ito_tanaka}
|\cos\Theta_t|-|\cos\Theta_0|
=\int_0^t \frac{v\sin^2\Theta_s}{R_s}\,ds
-\frac{\sigma^2}{2}\int_0^t |\cos\Theta_s|\,ds
+L_t-M_t.
\end{equation}

On   $ \{\tau = \infty \}$,  from $\sin^2\T_s \geq 1 - |\cos\T_s|$ and \eqref{eq:I_finite},
$$\langle{M}\rangle_t \geq \sigma^2(t - I_\infty) \to \infty,$$
so ${M}_t/\langle{M}\rangle_t \to 0$ a.s.\ by the Law of Large Numbers for continuous martingales \cite[Chap. 5, Ex. (1.16)]{revuzYor}

Rearranging \eqref{eq:ito_tanaka},
\begin{equation}\label{eq_fin}
\int_0^t\frac{v\sin^2\T_s}{R_s}\,ds = |\cos\T_t| - |\cos\T_0| + \tfrac{\sigma^2}{2}\int_0^t|\cos\T_s|\,ds - L_t + {M}_t.\end{equation}

Dividing \eqref{eq_fin} by $\langle M\rangle_t$ and letting $t\to\infty$, we
examine each term on the right-hand side:
\begin{enumerate}
\item $\dfrac{|\cos\T_t|-|\cos\T_0|}{\langle M\rangle_t}\to 0$, since the
numerator is bounded and $\langle M\rangle_t\to\infty$;
\item $\int_0^t|\cos\T_s|\,ds\le I_\infty<\infty$ is bounded, so this term
divided by $\langle M\rangle_t$ tends to $0$;
\item $-\dfrac{L_t}{\langle M\rangle_t}\le 0$, since $L_t\ge 0$;
\item $\dfrac{M_t}{\langle M\rangle_t}\to 0$.
\end{enumerate}
Taking $\limsup_{t\to\infty}$ on the right-hand side, terms (1), (2), (4) vanish
and term (3) is nonpositive, so
\[
\limsup_{t\to\infty}\frac{1}{\langle M\rangle_t}
\left(|\cos\T_t|-|\cos\T_0|+\tfrac{\sigma^2}{2}\int_0^t|\cos\T_s|\,ds-L_t+M_t\right)
\le 0.
\]
On the other hand, the left-hand side of \eqref{eq_fin} satisfies, using
$R_s\le R_0$ and $\langle M\rangle_t=\sigma^2\int_0^t\sin^2\T_s\,ds$,
\[
\frac{1}{\langle M\rangle_t}\int_0^t\frac{v\sin^2\T_s}{R_s}\,ds
\ge \frac{v}{\sigma^2 R_0}>0
\qquad\text{for all }t,
\]
hence its $\liminf$ as $t\to\infty$ is at least $v/(\sigma^2 R_0)>0$. Since the
two sides of \eqref{eq_fin} are equal, we obtain
\[
0<\frac{v}{\sigma^2 R_0}\le \liminf_{t\to\infty}(\text{LHS})
=\liminf_{t\to\infty}(\text{RHS})\le \limsup_{t\to\infty}(\text{RHS})\le 0,
\]
a contradiction. Therefore $\P(\tau=\infty)=0$, i.e.\ $\tau<\infty$ a.s.
%
%
%Dividing by $\langle{M}\rangle_t$ and letting $t \to \infty$:
%\begin{enumerate}
%\item left hand side: since $R_s\le R_0$,
%\[
%\int_0^t\frac{v\sin^2\T_s}{R_s}\,ds
%\ge \frac{v}{R_0}\int_0^t\sin^2\T_s\,ds
%=\frac{v}{\sigma^2 R_0}\,\langle M\rangle_t,
%\]
%so the ratio is $\ge \dfrac{v}{\sigma^2 R_0}>0$;
%\item $\frac{|\cos\T_t| - |\cos\T_0|}{\langle{M}\rangle_t} \ra 0$;
%\item $\int_0^t|\cos\T_s|\,ds \leq I_\infty < \infty$ is bounded, so this term divided by $\langle{M}\rangle_t$ converges  to 0;
%\item $-L_t \leq 0$, so $-\frac{L_t}{\langle{M}\rangle_t} \leq 0$.
%\item $\frac{{M}_t}{\langle{M}\rangle_t} \to 0$.
%\end{enumerate}
%Hence, the left hand side of \eqref{eq_fin} is strictly larger than $0$, while the right hand side is non- positive, resulting in a  contradiction.
\end{proof}
\medskip

\begin{thm}\label{Solution_pn}
Let $\mu_1,\mu_2$ be as in \eqref{eq_pn}. Then for every $N\geq 2$ and every $(r,\t)\in \R_+ \times \R$
there exists a unique strong solution to \eqref{state} in $[0,\tau]$  in the sense of Definition \ref{def:sol_tau}.
\end{thm}

\begin{proof}

We restrict to $N=2$; the case $N>2$ follows from the same argument. Indeed, the
angular drift $\mu_2(r,\theta)=-(N-1)v\sin(\theta)/r$ differs from the $N=2$ case
only by the factor $(N-1)\ge 1$, which multiplies the restoring term: the
attraction toward the critical angles $\{j\pi\}$ is therefore at least as strong
as for $N=2$. Concretely, in the estimates of Proposition~\ref{tau_f_sw} the
quantity $\int \mu_2$ is scaled by $(N-1)$, so the lower bounds forcing
$\Theta\to j\pi$ only improve, while the radial dynamics $\mu_1$ are unchanged.
The conclusion thus holds for every $N\ge 2$.

We first  prove that the conditions of Definition \ref{def:sol_tau} are satisfied.  Item 1. follows from Proposition \ref{Solution}, so it remains to prove Item 2. 

Since $R_t$ is non-increasing, it has a limit, so it is enough to prove that $\lim_{t\nearrow \tau} \T_t$ exists a.s. 
We will actually prove that $\lim_{t\nearrow \tau} \T_t \in \cup_j \{j\pi \}$. 
Observe that in the Proportional Navigation model  we do not have intervals in $[0,2\pi[$ where $\mu_2 \equiv 0$ so there is only one type of limiting behavior.

Using Proposition \ref{tau_f}, it is straightforward to adapt Item 1)--5) of the proof of Theorem \ref{Solution_sail} to obtain that $\P(\cup_{j\in \Z} L_j) = 1$, where 
\begin{align}
 L^1_j =  \left\{  \frac{ (2j-1)\pi}{2} < \liminf_{t\nearrow \tau} \T_t \leq \limsup_{t
\nearrow \tau} \T_t < \frac{(2j+1)\pi}{2}  \right\},
\end{align}
namely $L^1_0 = ]-\pim,\pim[$.
\smallskip

We now proceed as in  Item 7)  of the proof of Theorem \ref{Solution_sail}, in particular,  prove that if $L_j^1$ occurs, then $\lim_{t \ra \tau} = j \pi$.
Using the symmetries of the problem, we assume that $j=0$.

Let
\begin{align}
&L^1_{0,\ve} = \left\{ -\pim + \ve < \liminf_{t\nearrow \tau} \T_t  \leq \limsup_{t\nearrow \tau} \T_t< \pim -\ve \right\}, \\
&J^1_{0,m}= \left\{  -\frac{\pi}{2} +  \frac{m^{-\frac{1}{4}}}{2} <\T_{\tau_{\frac 1m}} <  \frac{\pi}{2} -  \frac{ m^{-\frac{1}{4}}}{2}  \right\} \label{eq_jm1}.
 \end{align}
then $L^1_0 = \cup_\ve L^1_{0,\ve}$.
Define also
 \begin{align}
 L^1_{0,\ve,m} = L^1_{0,\ve} \cap  J^1_{0,m},
 \end{align}
so that
\begin{align}
L^1_0 = \bigcup_{\ve >0} \bigcup_{n\in \N^\star} \bigcap_{m\geq n} L^1_{0,\ve,m}.
 \end{align}
Finally, set
\begin{align}
&H^1_m = \left\{ \sup_{\tau_{\frac 1m} \leq s \leq \tau_{\frac 1m} + \frac{4}{mv}} 
|B_s - B_{\tau_{\frac 1m}}| \leq \frac{m^{-\frac 14}}{2\sigma} \right\}, \label{h1_def}
\end{align}
%Item 7) of the proof of Theorem \ref{Solution0} is redundant since it considers the case $\t \in ]\pim,\pi[ \cup ]\pitq, 2\pi[$ where  $\mu_2(\t) = 0$, or, in ger

The reader is invited to compare $H^1_m$ (resp. $J^1_{0,m}$) with $H_m$ (resp. $J_{0,m}$) in \eqref{eq_hm} (resp. \eqref{eq_jm}).

Given the validity of the following facts, the conclusion follows as in Item 7) of the proof of Theorem \ref{Solution_sail}.
\begin{enumerate}
\item Fact 1. $\sum_{m=1}^\infty \P\left((H^1_m)^c\right) < \infty$, so for  $m$ large enough, $H^1_m$ occurs, where $H^1_m$ is introduced in \eqref{h1_def}.
 \item   Fact 2. On $H^1_m \cap L^1_{0,\ve,m}$,  $\tau \leq \tau_{\frac 1m} + \frac{4}{mv}$ and  $ \limsup_{t \nearrow \tau} \sup_{s\geq t} |\T_s| \leq   m^{-\frac 14} $.
\end{enumerate}

{\em Proof of Fact 1.} 
Follows from Lemma \ref{eq_h_bc}, with  $\kappa = 4$, $\alpha = \tfrac{1}{4}$, $\sigma$ as given.% Then $C = \frac{4\sigma}{\sqrt{v}}$.

\smallskip
{\em Proof of Fact 2.}  See Proposition \ref{tau_f_sw}. 

\smallskip
Pathwise uniqueness follows by the same argument used in Item 8) of the proof of Theorem \ref{Solution_sail}.
\end{proof}

We now show that, on the event that the $\T$ stays away from $\pm\pi/2$ and the 
noise increment is small ($H^1_m \cap L^1_{0,\ve,m}$), the time to hitting the target satisfies 
$\tau \leq \tau_{\frac 1m} + \frac{4}{mv}$, and $\T_t$ is driven to within $m^{-\frac 14}$ 
of $0$ as $t \nearrow \tau$. This is required to apply Item 7) of the proof of Theorem \ref{Solution_pn}.

Moreover, the generalization discussed in Remark \ref{rem_gamma} applies verbatim here, with $\frac{1}{\rho}$ replaced by the corresponding kernel.

\begin{prop}\label{tau_f_sw}
Let $\mu_1,\mu_2$ as in \eqref{eq_pn}. 
There exists $M_0$ such that, on $H^1_m \cap L^1_{0,\ve, m}$,
for all $m \geq M_0$, $\tau \leq \tau_{\frac 1m} + \frac{4}{mv}$, and moreover
\begin{equation}\label{lim_sup}
\limsup_{t\nearrow \tau}  \sup_{s\geq t} |\T_s|  \leq  m^{-\frac 14}.
\end{equation}
\end{prop}

\begin{proof}
Without loss of generality (see also first paragraph of the proof of Theorem \ref{Solution_pn}), we restrict to the case $N=2$.

Recall that on $J_{0,m}$,
$\Theta_{\tau_{\frac 1m}} \in \left] -\pim + \tfrac{m^{-\frac 14}}{2}, \pim-\tfrac{m^{-\frac 14}}{2}\right[$.

\medskip
\noindent\textit{Bootstrap setup.}
We first set
\begin{equation}\label{tau_star}
\tau^{*} := \tau \wedge \left(\tau_{\frac 1m} + \frac{4}{mv}\right),
\end{equation}
and carry out all estimates below on $[\tau_{\frac 1m}, \tau^{*}]$. On this interval we are 
inside the window of $H^1_m$, so
\begin{equation}\label{Hm_increment}
\sigma\, |B_s - B_{\tau_{\frac 1m}}| \leq \frac{m^{-\frac 14}}{2} 
\qquad \text{for all } s \in [\tau_{\frac 1m}, \tau^{*}].
\end{equation}
The estimates of Phases~1 and~2 will yield $\tau^{*} - \tau_{\frac 1m} < \frac{4}{mv}$ 
strictly; hence the minimum in \eqref{tau_star} is not attained at the cap, so 
$\tau^{*} = \tau$ and \eqref{Hm_increment} holds on all of $[\tau_{\frac 1m}, \tau]$. 
In particular every a~priori increment bound used below is a consequence of 
\eqref{Hm_increment}.

We split the proof into two phases.
\smallskip
\begin{enumerate}
\item 
\textit{Phase 1: escape from $|\t| > \piq$.}
Suppose $\Theta_{\tau_{\frac 1m}} \in [\piq, \pim - \tfrac{ m^{-\frac{1}{4}}}{2}[$.
The case $|\Theta_{\tau_{\frac 1m}}| \in [0, \piq[$ skips Phase~1, and the case 
$\Theta_{\tau_{\frac 1m}} \in ]-\pim + \tfrac{ m^{-\frac{1}{4}}}{2}, -\piq ]$ is symmetric.

Define $\tau^1 := \inf \{t\geq \tau_{\frac 1m} : \, \T_t \leq \piq\} \wedge \tau^{*}$.
For $s\in [\tau_{\frac 1m}, \tau^1]$, since $H^1_m \cap L^1_{0,\ve, m}$ occurs, 
$\T_s \in [\piq, \pim-\ve[$, which gives $\sin \T_s \geq \sin \tfrac \pi4 = \tfrac{1}{\sqrt 2}$, 
and $R_s \leq \tfrac 1m$ since $R$ is non-increasing from $R_{\tau_{\frac 1m}} = \tfrac1m$.
Using these together with \eqref{Hm_increment},
\begin{align}\label{t_bound0}
\Theta_{\tau^1} 
&= \Theta_{\tau_{\frac 1m}} - \int_{\tau_{\frac 1m}}^{\tau^1} \frac{v  \sin \Theta_s}{R_s} \, ds 
+ \sigma (B_{\tau^1} - B_{\tau_{\frac 1m} })  \\
& \leq  \Theta_{\tau_{\frac 1m}} - \frac{mv}{\sqrt 2}(\tau^1 - \tau_{\frac 1m}) 
+  \frac{m^{-\frac{1}{4}}}{2}.
\end{align}
Since $\Theta_{\tau^1} \geq \piq$ and $\Theta_{\tau_{\frac 1m}} - \piq < \piq$, rearranging gives
\begin{equation}\label{t1_bound}
\tau^1 - \tau_{\frac 1m} 
\leq \frac{\sqrt 2}{mv}\left(\Theta_{\tau_{\frac 1m}} - \piq + \frac{ m^{-\frac 14}}{2}\right)
\leq \frac{\sqrt 2}{mv}\left(\frac{\pi}{4} + \frac{ m^{-\frac 14}}{2}\right) 
< \frac{\sqrt 2}{mv} \cdot \frac{\pi}{3}
\end{equation}
for $m > M_0$, with $M_0$ large enough.

\smallskip

\item 
\textit{Phase 2: containment in $[-\piq, \piq]$.}

By Phase~1, $\T_{\tau^1} \in [-\piq,\piq]$. Let $r_1 := R_{\tau^1}$. Since $R$ is 
non-increasing on $[\tau^1, \tau^{*}]$, the radii attained there are exactly 
$q \in [R_{\tau^{*}}, r_1]$, and for such $q$ we have $\tau_q \in [\tau^1, \tau^{*}]$, 
so \eqref{Hm_increment} applies at $\tau_q$. For $q \in [R_{\tau^{*}}, r_1]$,
\begin{equation}\label{336q}
\T_{\tau_q} = \T_{\tau^1} + \int_{\tau^1}^{\tau_q} \mu_2(R_s,\T_s)\, ds + \s (B_{\tau_q} - B_{\tau^1}).
\end{equation}
On $[R_{\tau^{*}}, r_1]$ the map $q \mapsto \tau_q$ is a diffeomorphism onto 
$[\tau^1, \tau^{*}]$; since $R_{\tau_q} = q$, the Leibniz rule gives
\begin{equation}\label{tau_diff_1}
\frac{d R_{\tau_q}}{d q} 
= -v \cos(\T_{\tau_q}) \frac{d\tau_q}{dq} = 1,
\qquad\text{so}\qquad
\frac{d\tau_q}{dq} =-\frac{1}{v\, \cos(\T_{\tau_q}) },
\end{equation}
which is well-defined on $L^1_{0,\ve,m}$ because $\cos(\T_{\tau_q}) \neq 0$. Writing 
$\tau^\prime_q := \tfrac{d \tau_q}{dq}$ and changing variables $s = \tau_\rho$ in \eqref{336q},
\begin{equation}\label{eq_int}
\T_{\tau_q} = \T_{\tau^1} + \int_{r_1}^{q} \mu_2(R_{\tau_\rho},\T_{\tau_\rho}) \tau^\prime_\rho\, d\rho 
+ \s (B_{\tau_q} - B_{\tau^1}).
\end{equation}

By \eqref{Hm_increment}, for $q \in [R_{\tau^{*}}, r_1]$,
\begin{equation}\label{B_phase2}
|\s(B_{\tau_q} - B_{\tau^1})| 
\leq \sigma|B_{\tau_q} - B_{\tau_{\frac 1m}}| + \sigma|B_{\tau^1} - B_{\tau_{\frac 1m}}| 
\leq m^{-\frac 14}.
\end{equation}

Let $\tau^0_{r_1} := \inf \{t > \tau^1 : \T_{t} = 0 \}$.
\smallskip

\begin{enumerate}

\item 
\textit{Case $\tau^0_{r_1} \geq \tau^{*}$.}
Then $\T_t$ does not change sign on $[\tau^1,\tau^{*}]$. Assume 
$\T_{\tau^1} \in ]0, \piq]$ (the case $\T_{\tau^1} \in [-\piq, 0]$ is symmetric). From \eqref{eq_int}, using 
$R_{\tau_\rho} = \rho$ and \eqref{tau_diff_1}, for $q \in [R_{\tau^{*}}, r_1]$,
\begin{align}\label{eq_34}
\T_{\tau_q} 
&= \T_{\tau^1} - \int_{q}^{r_1} \frac{v\sin(\T_{\tau_\rho})}{\rho} \cdot \frac{1}{v\, \cos(\T_{\tau_\rho})}\, d\rho 
+ \s(B_{\tau_q} - B_{\tau^1})\\
&= \T_{\tau^1} - \int_{q}^{r_1} \frac{\tan(\T_{\tau_\rho})}{\rho}\, d\rho + \s(B_{\tau_q} - B_{\tau^1})\\
&\leq \T_{\tau^1} - \int_{q}^{r_1} \frac{\T_{\tau_\rho}}{\rho}\, d\rho + \s(B_{\tau_q} - B_{\tau^1}),
\end{align}
where the last step uses $\tan x \geq x$ on $[0,\piq]$. By \eqref{B_phase2},
\begin{equation}\label{eq_t3}
\T_{\tau_q} \leq  \T_{\tau^1} - \int_{q}^{r_1} \frac{\T_{\tau_\rho}}{\rho}\, d\rho + m^{-\frac 14},
\end{equation}
and Gronwall's inequality gives, for $q \in [R_{\tau^{*}}, r_1]$,
\begin{equation}\label{eq_gronwall}
\T_{\tau_q} \leq \bigl(\T_{\tau^1} + m^{-\frac 14}\bigr)\frac{q}{r_1} 
\leq \piq\,\frac{q}{r_1} + m^{-\frac 14}.
\end{equation}

\smallskip

\item 
\textit{Case $\tau^0_{r_1} < \tau^{*}$.}
Here $\T_t$ returns to $0$ before $\tau^{*}$, so $[\tau^1,\tau^{*}]$ decomposes into 
excursions of $\T$ away from $0$. Set $C := \{l \in [R_{\tau^{*}},r_1] :\, \T_{\tau_l} = 0 \}$. 
Fix any $q \in [R_{\tau^{*}},r_1] \setminus C$ and let 
$\xi_q := \inf \{ l \geq q :\, \T_{\tau_l} = 0 \}$. Then $\T_{\tau_\rho}$ keeps a constant 
sign on $[q,\xi_q[$ and $\T_{\tau_{\xi_q}} = 0$. Assume $\T_{\tau_q} > 0$ (the case 
$\T_{\tau_q} < 0$ is symmetric). By the same computation that gave \eqref{eq_t3}, applied 
on $[q,\xi_q]$ with $\T_{\tau_{\xi_q}} = 0$,
\begin{equation}\label{eq_t4}
\T_{\tau_q} \leq - \int_{q}^{\xi_q} \frac{\T_{\tau_\rho}}{\rho}\, d\rho + m^{-\frac 14} 
\leq  m^{-\frac 14}.
\end{equation}
Every $q \in [R_{\tau^{*}},r_1]$ either belongs to $C$, where $\T_{\tau_q} = 0$, or to an 
excursion, where \eqref{eq_t4} applies; in either case
\begin{equation}\label{eq_excursion_bound}
|\T_{\tau_q}| \leq m^{-\frac 14} \qquad \text{for all } q \in [R_{\tau^{*}},r_1].
\end{equation}

\end{enumerate}

Combining \eqref{eq_gronwall} and \eqref{eq_excursion_bound}, for all $q \in [R_{\tau^{*}},r_1]$
\begin{equation}\label{eq_uniform}
|\T_{\tau_q}| \leq \piq\,\frac{q}{r_1} + m^{-\frac 14}.
\end{equation}

\smallskip
\noindent\textit{Closing the bootstrap: $\tau^\star = \tau$.}

By \eqref{eq_uniform}, $\Theta_t \leq \piq + m^{-\frac 14} \leq \tfrac\pi3$ for $m\geq M_0$ large enough and for all 
$\tau_q \in [\tau^1, \tau^{*}]$ (i.e. $q \in [R_{\tau^{*}}, r_1]$), with no assumption 
on $\tau^{*}$. Hence $\frac{dR_t}{dt} = \mu_1(\Theta_t) = -v\cos\Theta_t \leq -\tfrac v2$ 
on $[\tau^1, \tau^{*}]$, so for $t \in [\tau^1, \tau^{*}]$,
\begin{equation}\label{t2_bound}
R_t \leq R_{\tau^1} - \frac v2 (t - \tau^1) \leq \frac1m - \frac v2 (t - \tau^1).
\end{equation}
Set $t_0 := \tau^1 + \frac{2}{mv}$, the time at which the right-hand side vanishes. 
By \eqref{t1_bound},
\begin{equation}\label{t0_window}
t_0 - \tau_{\frac 1m} = (\tau^1 - \tau_{\frac 1m}) + \frac{2}{mv} 
< \frac{\sqrt 2\,\pi}{3mv} + \frac{2}{mv} < \frac{4}{mv},
\end{equation}
so $t_0 < \tau_{\frac 1m} + \frac{4}{mv}$. 

We claim $\tau \leq t_0$. If instead 
$\tau > t_0$, then  $t_0 < \tau \wedge 
(\tau_{\frac 1m} + \frac{4}{mv}) = \tau^{*}$ by \eqref{t0_window} and $R_t > 0$ on $[\tau^1, t_0]$; thus \eqref{t2_bound} 
applies at $t_0$, giving $R_{t_0} \leq \frac1m - \frac v2 \cdot \frac{2}{mv} = 0$, 
contradicting $R_{t_0} > 0$. Hence $\tau \leq t_0$, and by \eqref{t0_window}
\begin{equation}\label{tau_bound1}
\tau - \tau_{\frac 1m} \leq t_0 - \tau_{\frac 1m}  < \frac{4}{mv}.
\end{equation}
In particular $\tau < \tau_{\frac 1m} + \frac{4}{mv}$, so $\tau^{*} = \tau$ and 
$R_{\tau^{*}} = R_\tau = 0$.

\smallskip
\noindent\textit{Conclusion of \eqref{lim_sup}.}
Now that $\tau^{*} = \tau$, the uniform bound \eqref{eq_uniform} holds for all $q \in [0, r_1]$ 
along $[\tau^1, \tau]$. Since $R_\tau = 0$, as $t \nearrow \tau$ we have 
$q = R_{\tau_q} \to 0$, so $\tfrac{q}{r_1} \to 0$ and
\begin{equation}\label{eq_c_1}
\limsup_{t\nearrow \tau} \sup_{s\geq t} |\T_s| \leq m^{-\frac 14},
\end{equation}
which is \eqref{lim_sup}.
\end{enumerate}
\end{proof}

\addcontentsline{toc}{section}{Bibliography}
\bibliographystyle{plain}
\bibliography{biblioCarlo}

\end{document}